\documentclass[preprint,10pt]{elsarticle}

\usepackage{amssymb}
\usepackage{amsmath}
\usepackage{amsthm}
\newtheorem{theorem}{theorem}[section]

\journal{Mathematics Journal}

\begin{document}

\begin{frontmatter}



\title{A method to identify the ordinary edges for symmetric traveling salesman problem based on frequency $K_i$s}


\author{Yong Wang} 

\affiliation{organization={New Energy School, North China Electric Power University},
            addressline={No.2, Beinong Road, Huilongguan, Changping}, 
            city={Beijing},
            postcode={102206}, 
            country={China}}

\begin{abstract}
The frequency $K_i$s ($i\in[4,n]$) are studied for symmetric traveling salesman problem ($TSP$) to characterize the structure properties of the edges inside and outside the optimal Hamiltonian cycle ($OHC$). Given a $K_i$ in  $K_n$ where $i\in [4,n]$, the frequency $K_i$ is computed with the set of ${{i}\choose{2}}$ optimal $i$-vertex paths with fixed endpoints (optimal $i$-vertex paths) in the $K_i$. 
Given an $OHC$ edge in a $K_i$, it has a frequency bigger than $\frac{1}{2}{{i}\choose{2}}$ in the frequency $K_i$, and that of an ordinary edge outside the $OHC$ is smaller than $\frac{1}{2}{{i}\choose{2}}$. 
As the frequency of an edge is computed with the frequency $K_i$s, an $OHC$ edge of $K_n$ has an average frequency bigger than $\frac{1}{2}{{i}\choose{2}}$. It indicates an $OHC$ edge of $K_n$ is also one $OHC$ edge of a $K_i$ containing it. It also found that the probability that an $OHC$ edge has the frequency bigger than $\frac{1}{2}{{i}\choose{2}}$ increases according to $i\in [4, n]$ based on the frequency $K_i$s. 
For an ordinary edge outside the $OHC$, the probability that it has a frequency smaller than $\frac{1}{2}{{i}\choose{2}}$ increases according to $i$. Based on the findings, a method is given to identify the ordinary edges for $TSP$. 
\end{abstract}

\begin{graphicalabstract}
\end{graphicalabstract}

\begin{highlights}
\item The lower frequency bound for $OHC$ edges in $K_i$ is derived.
\item The sufficient and necessary conditions for $OHC$ edges in $K_i$ are given. 
\item The difference between $OHC$ and ordinary edges is demonstrated. 
\end{highlights}

\begin{keyword}
Traveling salesman problem\sep optimal Hamiltonian cycle\sep optimal $i$-vertex path with given endpoints\sep frequency $K_i$

\MSC 05C45\sep 68Q06\sep 05C90\sep 90B40
\end{keyword}

\end{frontmatter}



\section{Introduction}
\label{sec1}
Traveling Salesman Problem ($TSP$) is extensively studied in combinatorial optimization, operations research, computer science and engineering \cite{DBLP:books/Gutin07}. 
Given $K_n$ on $n$ vertices $\{1, \ldots, n\}$, there is a distance $d(u,v)> 0$ for an edge $(u,v)$ where $u \neq v \in \{1, 2, \ldots, n\}$. For the symmetric $TSP$, $d(v,u)=d(u,v)$ exists for any pair of vertices $u$, $v$. A Hamiltonian cycle ($HC$) is one cycle visiting each of all vertices exactly once. The optimal Hamiltonian cycle ($OHC$) is the shortest one among all $HC$s. Finding $OHC$ is the classic $TSP$ and it is $NP$-Complete \cite{DBLP:journals/Karp75}. The number of $HC$s is $\frac{(n-1)!}{2}$ in symmetric $K_n$. As $n$ becomes big, it is impractical to find $OHC$ by the methods based on enumeration. 

The exact algorithms usually require $O(a^n)$ time for finding $OHC$ where $a>1$. For example, the dynamic programming
consumes $O(n^22^n)$ time owing to Bellman \cite{DBLP:journals/Bellman62}, and independently Held and Karp \cite{DBLP:journals/Karp62}. 
The techniques based on cutting-plane \cite{DBLP:journals/Levine00,DBLP:journals/Applegate09} 
and branch-and-bound \cite{DBLP:journals/Carpaneto95,DBLP:journals/Klerk11} 
are able to tackle $TSP$ with thousands of vertices. In 2006, one large Euclidean $TSP$ having 85,900 nodes was resolved  \cite{DBLP:journals/Applegate09}. 
In 2019, Cook \cite{DBLP:journals/Cook19} 
reported that the $OHC$ through 109,399 stars was found. However, the exact algorithms generally consume long time for resolving the $TSP$ instances of large scale.

$TSP$ is hard to resolve because $OHC$ edges do not show special distances in $K_n$. In other words, although $OHC$ edges are different from ordinary edges, it is difficult to separate them  according to the distances on edges. In the 1920s, Menger pointed out the nearest-neighbor method is not helpful for $TSP$ even though it works well to find the minimum spanning tree ($MST$) in a weighted graph \cite{DBLP:journals/Liu08}. In contrast, many ordinary edges have some special features based on distances. Forty years ago, Jonker and Volgenant \cite{DBLP:journals/Jonker84} identified more than half ordinary edges outside $OHC$ based on 2-$opt$ $moves$. Ten years ago, Hougardy and Schroeder \cite{DBLP:journals/Hougardy14} found more ordinary edges using 3-$opt$ $moves$. Zhong \cite{DBLP:journals/Zhong18} compared the two methods when they were applied to Euclidean $TSP$, and illustrated that the method based on the 3-$opt$ $moves$ identified more ordinary edges than that according to the 2-$opt$ $moves$. 

According to the distances on edges, the properties of $OHC$ edges are not fully  demonstrated in the above studies. In real-world applications, there are selection criteria for finding candidate $OHC$ edges. To reduce the search time of heuristic algorithms, a limited number of nearest neighbors (i.e., five nearest  neighbors) to each vertex are considered in the original LK algorithm \cite{DBLP:journals/Lin73}. The LK heuristic was improved to LKH-1 in the year of 2000. One of the essential improvements is the $\alpha$-measure computed for edges based on 1-tree. It was taken as the criteria for choosing a few candidate $OHC$ edges linking each vertex \cite{DBLP:journals/Helsgaun20}. The experiments illustrated that the  $\alpha$-measure was much better than the method of nearest neighbors for choosing the candidate $OHC$ edges, and it also played well in the new version LKH-2  \cite{DBLP:journals/Helsgaun09}. Besides the $\alpha$-measure, pseudo-backbone edges were proposed for finding the candidate $OHC$ edges contained in a set of high quality tours \cite{DBLP:journals/Jäger14}. A backbone edge is originally contained in every $OHC$ and it is relaxed to one pseudo-backbone edge. The tours of high quality can be computed with the heuristic algorithms, such as LKH. 
It is interesting that the high quality tours contain many pseudo-backbone edges. Through contracting the pseudo-backbone edges, the size of $TSP$ is greatly reduced. Although the $\alpha$-measure and pseudo-backbone edges work well for most $TSP$ instances, they are lack of theoretical foundations for characterizing the $OHC$ edges and ordinary edges. 

In recent years, the frequencies of edges were computed with the optimal 4-vertex paths with given endpoints (optimal 4-vertex paths), or frequency $K_4$s for identifying $OHC$ edges \cite{DBLP:journals/Wang151, DBLP:journals/Wang16}. Given a $K_4$ on four vertices in $K_n$, there are twelve 4-vertex paths each of which visits the four vertices exactly once in the $K_4$. They include six pairs of 4-vertex paths each of which has the same two endpoints. Each pair of the 4-vertex paths are compared with respect to the distance, and the shorter one is taken as the optimal 4-vertex path for the pair of specific endpoints, respectively. Since six pairs of vertices exist in the $K_4$, there are six optimal 4-vertex paths in the $K_4$. After the six optimal 4-vertex paths were obtained, the number of the optimal 4-vertex paths containing each edge is enumerated, respectively. Each number denotes the frequency of one corresponding edge. After the frequencies of all edges are figured out, the frequency $K_4$ is constructed by the edges and their frequencies, see Appendix $A$. The frequency $K_4$s for a $K_4$ with various distances on edges have been studied in paper \cite{DBLP:journals/Wang16}. 

Given an edge in $K_n$, it is contained in ${{n-2}\choose{2}}$ $K_4$s, and it is contained in the same number of frequency $K_4$s. In each frequency $K_4$, the edge has certain frequency of 1, 3, or 5. Given an edge in $K_n$, all the frequencies related to the edge are added together for computing a total frequency in  $\left[{{n-2}\choose{2}}, 5{{n-2}\choose{i-2}}\right]$ and the average frequency in $[1,5]$. 
After that, it found that the average frequency of an $OHC$ edge is much higher than the average frequency of all edges and that of most ordinary edges. In the average case, the average frequency of an $OHC$ edge is bigger than $3$ \cite{DBLP:journals/Wang16}. Moreover, an $OHC$ edge will be contained in more percentage of the optimal 4-vertex paths according to rising $n$ \cite{DBLP:journals/Wang19}. It implies that the average frequency of an $OHC$ edge will increase according to $n$ based on the frequency $K_4$s. For big and large $TSP$, most $OHC$ edges have a very big frequency computed with the frequency $K_4$s.

The frequency $K_4$s have been investigated, and each $OHC$ edge in $K_n$ has certain (average) frequency much higher than that of most ordinary edges based on the frequency $K_4$s. The next work is to study the frequency $K_i$s ($4<i\leq n$) for $TSP$. In one frequency $K_i$, the frequency of each edge is computed with the ${{i}\choose{2}}$ optimal $i$-vertex paths contained in one  kind of $K_i$s, see Section 2. In each  frequency $K_i$, every edge has a frequency in $\left[0,{{i}\choose{2}} - 1\right]$. An edge in $K_n$ is contained in ${{n-2}\choose{i-2}}$ frequency $K_i$s. The frequency or average frequency of each edge can be computed based on these frequency $K_i$s. Similar to the frequency of edges computed based on the frequency $K_4$s, $OHC$ edges and ordinary edges will have much difference with respect to their frequencies or average frequencies computed with the  frequency $K_i$s. According to the distance on edges in $K_n$, it is difficult to separate $OHC$ edges from ordinary edges. Since $OHC$ edges and ordinary edges have much difference with respect to the frequencies computed with the frequency $K_i$s, they may be separated from each other according to their frequencies.


Different from the previous research on frequency $K_i$s for $TSP$, this paper focuses on the intrinsic difference between $OHC$ edges and ordinary edges with respect to edge frequency computed based on frequency $K_i$s. If each $K_i$ contains ${{i}\choose{2}}$ optimal $i$-vertex paths and one $OHC$ where $i\in[4,n]$, this paper presents the following  results. 
An $OHC$ edge of a $K_i$ has the frequency bigger than $\frac{1}{2}{{i}\choose{2}}$ in the corresponding frequency $K_i$, i.e., an $OHC$ edge is contained in bigger than  $\frac{1}{2}{{i}\choose{2}}$ optimal $i$-vertex paths in the $K_i$. On the other hand, an ordinary edge has the frequency smaller than $\frac{1}{2}{{i}\choose{2}}$ in the frequency $K_i$, and the expected frequency of an ordinary edge is smaller than $\frac{i+2}{2}$. On average, the expected frequency of an $OHC$ edge in the $K_i$ is bigger than $\frac{i^2-4i+7}{2}$, and an ordinary edge has the expected frequency smaller than 2. Thus, the sufficient and necessary condition for an $OHC$ edge in a $K_i$ is that the edge frequency is bigger than $\frac{1}{2}{{i}\choose{2}}$ based on the ${{i}\choose{2}}$ optimal $i$-vertex paths in the $K_i$. 

As the average frequency of an edge in $K_n$ is computed with the frequency $K_i$s, an $OHC$ edge in $K_n$ has the lower frequency bound $\frac{1}{2}{{i}\choose{2}}$. It indicates that it has a frequency bigger than $\frac{1}{2}{{i}\choose{2}}$ in each frequency $K_i$ containing it in the average case. Thus, it is the $OHC$ edge of any $K_i$ containing it. In addition, the probability that an $OHC$ edge has a frequency bigger than $\frac{1}{2}{{i}\choose{2}}$ increases according to $i$ based on the frequency $K_i$s. On the other hand, the probability that an ordinary edge has a frequency smaller than $\frac{1}{2}{{i}\choose{2}}$ increases according to $i$ based on the frequency $K_i$s. According to the findings, a method is given to find the ordinary edges for $TSP$ based on the frequency $K_i$s. 

The remainder of this paper is organized as follows. In Section 2, the optimal $i$-vertex paths contained in a $K_i$ are introduced, and the frequency $K_i$ is computed with the optimal $i$-vertex paths. In Section 3, the lower frequency bound for $OHC$ edges in a $K_i$ is proven, and the upper frequency bound for ordinary edges is also derived. 
In section 4, the change of the probability that an edge has a frequency bigger than $\frac{1}{2}{{i}\choose{2}}$ according to $i$ is studied for $OHC$ edges and ordinary edges, respectively, based on the frequency $K_i$s. Meanwhile, the method to identify ordinary edges is also presented for $TSP$. 
Finally, conclusions are drawn in the last section, and the future research is also presented. 


\section{The optimal $i$-vertex paths and frequency $K_i$}
\label{sec2}

  Given $K_n$ on $n$ vertices in  $V=\{{v_1, v_2,...,v_n}\}$ for symmetric $TSP$, each vertex $v_i$ is assigned by a natural number in $[1, n]$ where $i\in [1,n]$. For any set of vertices in $K_n$, there is a total order on the set of vertices induced by the natural ordering on ${1, 2, ..., n}$. Based on the unique natural number assigned to each vertex, all the edges, paths and $HC$s in $K_n$ are denoted by different vertex sequences.  
The $TSP$ containing one $OHC$ is considered here. There are ${{i}\choose{2}}$ edges in $K_n$, and there are $n$ $OHC$ edges. \textit{If an edge is contained in $OHC$, it is called an $OHC$ edge. Otherwise, it is an ordinary edge.}  If some $TSP$ includes several or many $OHC$s, one can add small random distances to the distances of edges, and it is converted into one modified $TSP$ having only one $OHC$ and $n$ $OHC$ edges. As the random distances are small enough, the original $TSP$ and modified $TSP$ will contain the same $OHC$. The $OHC$ in $K_n$ is illustrated in Figure \ref{OHC} where we assume that the subscripts of vertices are ordered according to the natural numbers $1, 2, ..., n$. The dashed lines are $OHC$ edges whereas the solid line represents an ordinary edge $(v_1,v_j)$ and $1<i<j<k<n$.

\begin{figure}
	\centering
	\includegraphics[width=2.5in,bb=0 0 400 200]{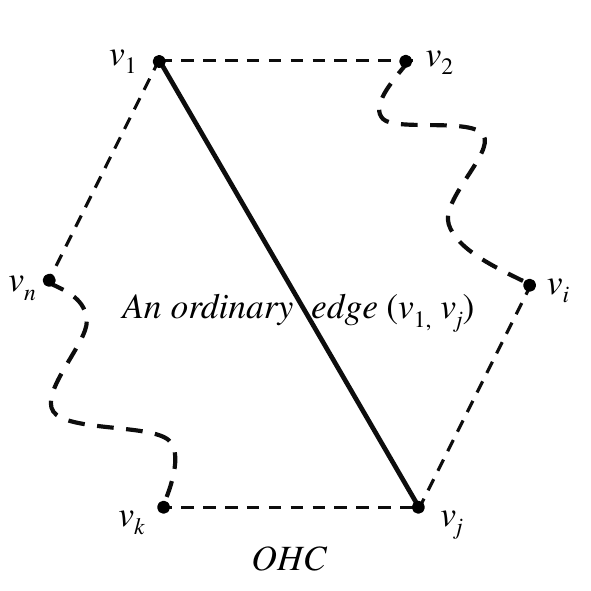}
	\caption{An ordinary edge and the two pairs of adjacent $OHC$ edges in $K_n$.}
	\label{OHC}
\end{figure}

In $K_n$, each vertex is contained in $n-1$ edges. For example, $v_1$ is contained in the $n-1$ edges $(v_1, v_x)$ where $x\in [2,n]$ in Figure \ref{OHC}. Since there is only one $OHC$ in $K_n$, each vertex is contained in two $OHC$ edges and $n-3$ ordinary edges, respectively. For example in Figure \ref{OHC}, $v_1$ is contained in the two $OHC$ edges $(v_1,v_2)$ and $(v_1, v_n)$, and the $n-3$ ordinary edges $(v_1,v_y)$ $(y\in [3,n-1])$. 

Moreover, each ordinary edge is adjacent to two pairs of $OHC$ edges on both endpoints since each vertex is contained in two $OHC$ edges. For example in Figure \ref{OHC}, the ordinary edge $(v_1,v_j)$ is adjacent to the two pairs of $OHC$ edges $(v_1,v_2)$ $\&$ $(v_1,v_n)$ and $(v_j,v_i)$ $\&$ $(v_j, v_k)$. If $v_2 = v_i$ or $v_k = v_n$ for $n >4$, there are at most $n$ such ordinary edges in $K_n$. For each of the other $\frac{n(n-5)}{2}$ ordinary edges $(v_1, v_j)$, $v_2\neq v_i$ and $v_k\neq v_n$ exist. It means that most of the ordinary edges $(v_1,v_j)$ and the two pairs of adjacent $OHC$ edges lie within a designated hexagon instead of a pentagon. In the following, for an ordinary edge $(v_1, v_j)$, the four vertices $v_2$, $v_i$, $v_k$ and $v_n$ contained in the two pairs of adjacent $OHC$ edges are pairwise distinct if there are no special declarations. 

Given a set of $i$ vertices in $K_n$ where $i\in[4,n]$, the $i$ vertices are contained in one  corresponding $K_i$. As there is only one $OHC$ in the $K_i$, every vertex is contained in two $OHC$ edges and $i-3$ ordinary edges, and an ordinary edge is adjacent to two pairs of $OHC$ edges on both endpoints, respectively. Moreover, there are ${{i}\choose{2}}$ optimal $i$-vertex paths in the $K_i$. The corresponding frequency $K_i$ is computed with the optimal $i$-vertex paths ($OP^i$) in the $K_i$ where the superscript $i$ in $OP^i$ denotes the number of vertices. The optimal 4-vertex paths and frequency $K_4$s were computed in paper \cite{DBLP:journals/Wang16}, see Appendix $A$. The optimal $i$-vertex paths in the $K_i$ and the corresponding frequency $K_i$ are computed as below. 

\begin{description}
	\item[Optimal $i$-vertex path ($OP^i$)]: Given a $K_i$ on a set of $i$ vertices $\{v_1,v_2, \ldots, v_{i-1},v_i \}$ in $K_n$ where $i\in [4,n]$, the paths $P^i = (v_{\sigma_1}, \ldots, v_{\sigma_i})$ visit each of the $i$ vertices exactly once. Fix the two endpoints of $P^i$, such as $v_{\sigma_1}=v_1$
	and $v_{\sigma_i} =v_2$, there are $(i-2)!$ $P^i$s where $v_1$ and $v_2$ are the endpoints. The shortest one is taken as the optimal $i$-vertex path denoted as $OP^i$ for $v_1$ and $v_2$. 
\end{description}
In the $K_i$, there are ${{i}\choose{2}}$ pairs of vertices which are taken as the endpoints for computing the $OP^i$s. Thus, each $K_i$ contains ${{i}\choose{2}}$ $OP^i$s. The ${{i}\choose{2}}$ $OP^i$s in the $K_i$ on vertex set $\{v_1,v_2, \ldots, v_{i-1},v_i \}$ are illustrated in Figure \ref{OPi}. Each $OP^i$ has two specified endpoints which are different from those of the other $OP^i$s. The endpoints of each $OP^i$ are actually contained in one corresponding edge in the $K_i$. Thus, the number of $OP^i$s equals to that of edges in the $K_i$. For example, the first $OP^i$ in Figure \ref{OPi} has the endpoints $v_1$ and $v_2$. In the $K_i$, $v_1$ and $v_2$ are contained in the edge $(v_1,v_2)$. In this case, $(v_1,v_2)$ will not be contained in the first $OP^i$ in Figure \ref{OPi}. Thus, an edge in the $K_i$ is excluded from at least one $OP^i$. For each vertex in the $K_i$, such as $v_1$ in Figure \ref{OPi}, it is one endpoint of $i-1$ $OP^i$s. In each of these $OP^i$s, $v_1$ is contained in one edge. In each of the other ${{i-1}\choose{2}}$ $OP^i$s, $v_1$ is one intermediate vertex, and it is contained in two edges. The edges containing $v_1$ in the $OP^i$s are one subset of the edges containing $v_1$ in the $K_i$. 
It mentions that the $OHC$ in the $K_i$ contains $i$ $OP^i$s which only contain the $OHC$ edges. On the other hand, each of the other $\frac{i(i-3)}{2}$ $OP^i$s contains at least one ordinary edge, respectively. Most importantly, each $OP^i$ contains the defined set of $i-1$ edges in the $K_i$, and each $OP^i$ does not become shorter whatever the intermediate vertices are exchanged. Thus, all paths contained in an $OP^i$ are the optimal paths having a relatively smaller number of vertices. For example, each $P^k$ ($k\in [4,i]$) contained in an $OP^i$ is one $OP^k$.  

\begin{figure}
	\centering
	\includegraphics[width=3in,bb=0 0 360 250]{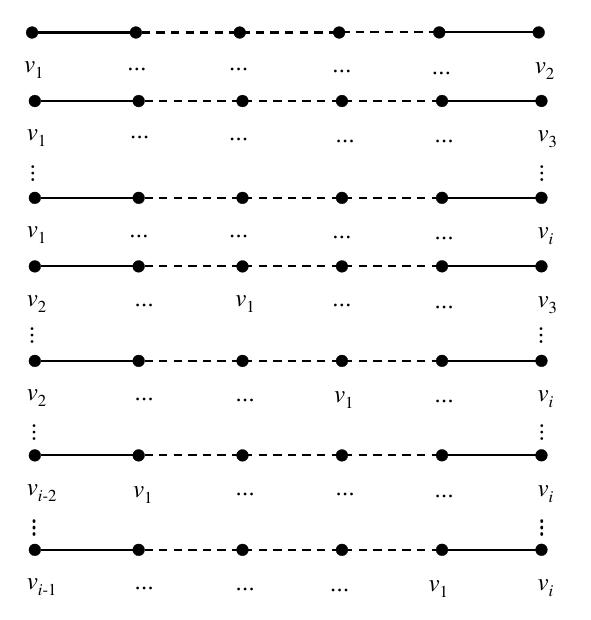}
	\caption{The ${{i}\choose{2}}$ $OP^i$s in one $K_i$ on vertex set $\{v_1,v_2,...,v_i\}$.}
	\label{OPi}
\end{figure}

In $K_n$, there are ${{n}\choose{i}}$ $K_i$s, and there are total ${{i}\choose {2}}{{n}\choose {i}}$ $OP^i$s at given number $i$.  All the $P^i$s in $OHC$ and $OP^n$s of $K_n$ are $OP^i$s contained in some $K_i$s. Moreover, every path $P^k$ contained in an $OP^m$ is one $OP^k$ where $k\in[4,m], m\in[k,n]$. If one $P^k$ becomes shorter by exchanging the intermediate vertices, the $P^m$ containing it must not be optimal. In addition, each path $P^k$ contained by two optimal paths $OP^i$ and $OP^j$ is also one $OP^k$ where $k\in [4, min\{i,j\}]$. Thus, the set of $OP^i$s and the intersection operation defined for any two optimal paths form one semi-group \cite{DBLP:journals/Wang251}. Due to the structure stability of the $OP^i$s, the $OP^i$s containing more vertices only include the $OP^i$s containing less vertices. The $OP^i$s containing the smaller number of vertices maintain the same structures when they are used to construct the $OP^i$s containing more vertices. 

The $OP^i$s in the $K_i$ have another property which does not become shorter through $k$-$opt$ $moves$ or edges replacement. Given an $OP^i$, it indicates that a pair of edges in the $OP^i$ can not be replaced by any pair of the other edges to compute another even shorter $OP^i$ in the $K_i$. Thus, each $OP^i$ will contain the pairs of $OHC$ edges in the $K_i$ as most as possible, respectively. On the other hand, each $OP^i$ contains the least number of ordinary edges in the $K_i$. As the ${{i}\choose{2}}$ $OP^i$s are obtained, the frequency $K_i$ is computed as follows. 
\begin{description}
	\item[Frequency $K_i$]: After the ${{i}\choose{2}}$ $OP^i$s are computed from a given $K_i$, one can enumerate the number of $OP^i$s containing an edge in the $K_i$. This number is called the frequency of the edge. As the frequency of each edge is computed based on these $OP^i$s, the frequency $K_i$ is constructed with the edges and their frequencies.
\end{description}
In the frequency $K_i$, the vertices and edges are the same as those in the corresponding $K_i$. In contrast, the distances of edges in the $K_i$ are replaced by the corresponding frequencies of the edges, see the frequency $K_4$s in Appendix $A$. Thus, one frequency $K_i$ and the corresponding $K_i$ contain the same $OHC$ and $OP^i$s. Different from the various distances of edges in the $K_i$, each edge has a specific frequency in $\left[0,{{i}\choose{2}}-1\right]$. If an edge is not contained in any one $OP^i$, it has the frequency of zero. Since an edge in the $K_i$ is excluded from at least one $OP^i$, the maximum frequency of edges is ${{i}\choose{2}} - 1$. Moreover, an $OHC$ edge has certain  frequency much higher than that of an ordinary edge in the frequency $K_i$. There are ${{i}\choose{2}}$ $OP^i$s in the $K_i$ and each $OP^i$ contains $i-1$ edges. Since the frequency $K_i$ contains ${{i}\choose{2}}$ edges, the average frequency of all edges is $i-1$. Given a vertex in the $K_i$, there are $i-1$ $OP^i$s in which it is one endpoint. In each of these $OP^i$s, the vertex is contained in one edge. In each of the other ${{i-1}\choose{2}}$ $OP^i$s, the vertex is one intermediate vertex, and it is contained in two edges. Thus, the total frequency of the $i-1$ edges containing the vertex is $(i-1)^2$. This indicates that the average frequency of the edges containing a vertex is also $i-1$. The distances on the $i-1$ edges containing a vertex do not have such properties. As the six frequency $K_4$s for a $K_4$ with various distances on the edges, the number of frequency $K_i$s are also enumerable.  

When one special $K_i$ contains many equal-weight edges, it may contain more than ${{i}\choose{2}}$ $OP^i$s and one $OHC$. It is difficult to choose a set of ${{i}\choose{2}}$ $OP^i$s from a lot of $OP^i$s to compute a unique frequency $K_i$. In this case, one will see that an edge may have various frequency based on different sets of $OP^i$s. Moreover, the frequency of ordinary edges and $OHC$ edges will have the same frequency according to certain set of $OP^i$s. Since we assume the $OHC$ edges are different from ordinary edges, they will have different structure characteristics to compute the $OP^i$s, and they are contained in different number of $OP^i$s. In other words, they should have different frequency based on the $OP^is$ in the $K_i$. However, the equal-weights make the $OHC$ edges and ordinary edges have no difference based on the $OP^i$s. Since we assume there is one $OHC$, the $OHC$ edges are definitely different from ordinary edges. It is the set of ${{i}\choose{2}}$ $OP^i$s that may discriminate the $OHC$ edges from ordinary edges in each $K_i$. If there are many equal-weight edges in a $K_i$, it must be converted into one $K_i$ containing only ${{i}\choose{2}}$ $OP^i$s. Then, the $OHC$ edges and ordinary edges may be identified based on their frequencies. One can add a set of small random distances to the distances of edges. As the random distances are sufficiently small, there exists at least one $OHC$ that remains unchanged. In this case, the modified $K_i$ will contain one $OHC$ and ${{i}\choose{2}}$ $OP^i$s. 

An edge $e$ in $K_n$ is contained in ${{n-2}\choose{i-2}}$ $K_i$s, and it is contained in the same number of frequency $K_i$s. In each frequency $K_i$, $e$ has certain frequency in  $\left[0,{{i}\choose{2}}-1\right]$. The (total) frequency and average frequency of $e$ is computed with these frequency $K_i$s.
\begin{description}
	\item[The (total) frequency and average frequency of an edge in $K_n$]:  For an edge $e$ in $K_n$, it is contained in $L_i={{n-2}\choose{i-2}}$ frequency $K_i$s. Provided that the frequency of $e$ is $f_k(e) (k\in [1,L_i])$ in the $k^{th}$ frequency $K_i$, the (total) frequency of $e$ is computed as $F(e) = \sum_{k=1}^{L_i} f_k(e)$ and the expected or average frequency of $e$ is  $f(e)=\frac{F(e)}{L_i}\in \left[0,{{i}\choose{2}}-1\right]$. 
\end{description}

The frequency or average frequency of $OHC$ edges has much difference from that of ordinary edges as they are computed with the frequency $K_i$s where $i\in[4,n]$. The frequency of an $OHC$ edge is much bigger than that of most ordinary edges computed based on frequency $K_4$s \cite{DBLP:journals/Wang19}. The $OHC$ edges and most ordinary edges can be separated from each other according to edge frequency. Besides the frequency, the probability that an edge is contained in the $OP^i$s also illustrates the structure characteristics of the edges with respect to $OHC$. In other words, an $OHC$ edge will be contained in the $OP^i$s with the big probability, whereas most ordinary edges are contained in the $OP^i$s with a small probability. 

\begin{description}
	\item[The probability that an edge is contained in the $OP^i$s]:  For an edge $e$ in $K_n$,  the average frequency of $e$ is 
	$f(e)\in \left[0,{{i}\choose{2}}-1\right]$ based on the frequency $K_i$s containing it. The expected probability that $e$ is contained in the $OP^i$s (or an $OP^i$) is computed as $p_i(e\in OP^i) = \frac{f(e)}{{{i}\choose{2}}}\in [0,1)$. If the $OP^i$ having the same endpoints as those of $e$ is neglected, $p_i(e\in OP^i) = \frac{f(e)}{{{i}\choose{2}} - 1}\in [0,1]$. 
\end{description}

For convenience, the expected probability is simply called probability in this paper. In view of the definitions of average frequency and probability for an edge, one will see that they are depends on the distances of all edges in $K_n$. Once the weighted $K_n$ for a $TSP$ instance is given, every edge has the unique average frequency and probability, respectively. If one takes the distance of an edge as the real function defined on the edge, the average frequency (or probability) is a functional defined on the function space related to the edge distances. Compared with the edge distances, the frequency or probability for an edge will exhibit the even stronger capability in characterizing the graph structure in terms of $OHC$. Therefore, we can build the sufficient and necessary conditions for $OHC$ edges and ordinary edges based on their frequencies and probabilities rather than distances. Then, the strategies or techniques based on these theories can be constructed to lead the algorithms to search for the $OHC$. 

In the following, as a $K_i$ ($i\in [4,n]$) or $K_i$s and frequency $K_i$ or frequency $K_i$s are discussed, they are totally contained in one $K_n$ for a given $TSP$ instance, and the distance of every edge is defined in the same space, such as one kind of Euclidean space, or metric space, etc. Once the $TSP$ instance is given, the $OP^i$s in each $K_i$ are determined according to the vertices representation and the distances on edges in the $K_i$, respectively. Meanwhile, the $OP^i$s containing each edge are also fixed according to the $K_n$. It means that a given edge $e$ has the specific $f(e)$ and $p_i(e\in OP^i)$ for the $TSP$.

\section{The lower frequency bound for $OHC$ edges and upper frequency bound for ordinary edges}
\label{sec3}
We assume that each $K_i$ $(i\in [4,n])$ contains one $OHC$ and ${{i}\choose{2}}$ $OP^i$s. Given a $K_i$, each $OP^i$ in the $K_i$ is unique with respect to the two  endpoints. Thus, each $OP^i$ in the $K_i$ contains the defined set of $i-1$ edges, respectively. Some edges are contained in many $OP^i$s, some edges are contained in a small number of $OP^i$s, and the remainder edges exclude from any one $OP^i$. Thus, not all edges in the $K_i$ are contained in the $OP^i$s, and the frequencies of  edges are different in the frequency $K_i$. Given a frequency $K_i$, the frequency of an $OHC$ edge is concerned. Moreover, as the average frequency and probability of an edge in $K_n$ are computed based on frequency $K_i$s, the difference between the average frequency and probability of $OHC$ edges and those of ordinary edges are also regarded. In the following, we first study the average frequency and probability for an $OHC$ edge in $K_n$ based on frequency $K_4$s. Given a $K_4$ on four vertices $A$, $B$, $C$, and $D$ in $K_n$, we assume the order to be $A<B<C<D$. For example, the $K_4$ on $A$, $B$, $C$, and $D$ is shown in Figure \ref{quadrilaterals} (a) where the distances on edges are not shown.

\begin{theorem}
	\label{th001}
	Given an edge $e\in OHC$ in $K_n$, the average frequency $f(e)\geq 3$ and the probability $p_4(e\in OP^4)\geq \frac{1}{2}$ hold based on the frequency $K_4$s. 
\end{theorem}

\begin{proof}
	If $i=4$, the six frequency $K_4$s for a $K_4$ with various distances on edges were studied in  \cite{DBLP:journals/Wang16}.  For convenience, a frequency $K_4$ on four vertices $A$, $B$, $C$ and $D$ is also called frequency $ABCD$. The six frequency $ABCD$s for an $ABCD$ are illustrated in Figure \ref{quadrilaterals}.
	
	Given a $K_4$ on four vertices $A$, $B$, $C$ and $D$ in $K_n$, there are six edges $(A,B)$, $(A,C)$, $(A,D)$, $(B,C)$, $(B,D)$ and $(C,D)$. The distances of the six edges are noted as $d(A,B)$, $d(A,C)$, $d(A,D)$, $d(B,C)$, $d(B,D)$ and $d(C,D)$, respectively. Moreover, the $K_4$ contains six $OP^4$s. It is known that the six $OP^4$s in $ABCD$ are determined by the order of the three distance sums $S_1 = d(A,B) + d(C,D)$, $S_2 = d(A,D) + d(B,C)$, and $S_3 = d(A,C) + d(B,D)$ \cite{DBLP:journals/Wang16}. For example, if $S_1 < S_2 < S_3$, the six $OP^4$s in $ABCD$ are derived as $(A,B,C,D)$, $(A,D,D,A)$, $(C,D,A,B)$, $(D,A,B,C)$, $(A,B,D,C)$, and $(B,A,C,D)$. The frequency $K_4$ is computed as shown in Figure \ref{quadrilaterals} (a). For $S_1$, $S_2$ and $S_3$, there are six orderings. One set of six $OP^4$s is computed based on each of the six orderings, respectively. Thus, there are six sets of six $OP^4$s which are used to compute six frequency $K_4$s shown in Figure \ref{quadrilaterals} (a)$\sim$(f) where the order of $S_1$, $S_2$ and $S_3$ in the corresponding $K_4$ is also illustrated under each frequency $K_4$, respectively. Although the distances of the six edges in $ABCD$ have infinite combinations, $S_1$, $S_2$ and $S_3$ only have six orderings. Thus, there are only six frequency $ABCD$s for all weighted $ABCD$s. Given a $K_4$ in $K_n$, the corresponding frequency $K_4$ may be any one of the six cases in Figure \ref{quadrilaterals} if the edge distances are not given. 
	
	Given an edge in a frequency $K_4=ABCD$, such as $(A,B)$, it has the frequency 1, 3, or 5. As $(A,B)$ is the $OHC$ edge in a $K_4$, it has a frequency 3 or 5 rather than 1 in the corresponding frequency $K_4$. Based on the six frequency $K_4$s in \ref{quadrilaterals} (a)$\sim$(f), $(A,B)$ has the frequency 5, 5, 3, 3, 1, and 1, respectively. Moreover, $(A,B)$ has each of the frequencies twice, respectively. Let $p_1(e)$, $p_3(e)$, and $p_5(e)$ denote the probability that $e=(A,B)$ has the frequency 1, 3, and 5, $p_1(e)=p_3(e)=p_5(e) = \frac{1}{3}$ holds based on the six frequency $K_4$s. 
	
	For an edge in a pair of $K_4$ and frequency $K_4$, the frequency and the related distance sum, form a one-to-one mapping relationship. For example $(A,B)$, if it has the biggest frequency 5 in a frequency $ABCD$, $S_1 < S_2$ and $S_1 < S_3$ must hold in the corresponding $ABCD$ and vice versa. On the other hand, if $S_1 > S_2$ and $S_3$ in $ABCD$, $(A,B)$ must have the smallest frequency 1 in the corresponding frequency $ABCD$ and vice versa. This bijection between the frequency of an edge in a frequency $K_4$ and the related distance sum in the corresponding $K_4$ will be used to derive a lower frequency bound for an $OHC$ edge in $K_n$ computed based on frequency $K_4$s. 
	
	The $OHC$ of $K_n$ is illustrated in Figure \ref{OHCKn} where $(A,B)$ is an $OHC$ edge, and $(C,D)\in OHC$ is one vertex-disjoint edge of $(A,B)$. The four vertices $A$, $B$, $C$ and $D$ are contained in one $ABCD$. Because $(A,B)$ and $(C,D)$ are two $OHC$ edges and  there is only one $OHC$ in $K_n$, $d(A,B) + d(C,D) < d(A,C) + d(B,D)$ must hold, i.e., $S_1 < S_3$. Otherwise, $(A,B)$ and $(C,D)$ will be replaced by $(A,C)$ and $(B,D)$ to compute another $HC$ even shorter than the $OHC$. Since $S_1 < S_3$ exists in $ABCD$, there are three frequency $K_4$s (a), (b) and (c) in Figure \ref{quadrilaterals}. In the three frequency $K_4$s (a), (b) and (c) in Figure \ref{quadrilaterals}, $(A,B)$ has the frequency 5, 5 and 3, respectively. It indicates that $(A,B)$ is contained in 5, 5, and 3 $OP^4$s in the three $K_4$s, respectively. It means that $e=(A,B)$ is contained in at least three $OP^4$s in the $ABCD$. Thus, $f(A,B) \geq 3$ and $p_4(e\in OP^4) \geq \frac{1}{2}$ hold based on the three frequency $K_4$s. 
	
	Under the distance sum inequality $S_1 < S_3$, the frequency $K_4$ for $ABCD$ may be any one of the three cases (a), (b) and (c) in Figure \ref{quadrilaterals} due to the variations of edge distances. Provided that the three frequency $K_4$s occur with the equal probability for $(A,B)$, $p_1(e) = 0$, $p_3(e) = \frac{1}{3}$, and $p_5(e) = \frac{2}{3}$ exist. In the average case, $f(A,B) = \frac{13}{3} > 3$ and $p(e\in OP^4) = \frac{13}{18} > \frac{1}{2}$ hold based on the three frequency $K_4$s (a), (b) and (c) in Figure \ref{quadrilaterals}. 
	
	Moreover, there are $n-3$ vertex-disjoint edges $(C,D)$ inside the $OHC$ for $(A,B)$. $S_1 < S_3$ holds in each of the $ABCD$s. Let a random variable $X$ representing the number of frequency $ABCD$s where $(A,B)$ has the frequency 3. As $X = k$ for $k\in [0,n-3]$, the probability conforms to the binomial distribution (\ref{eq1}). $P(X=k)$ monotone increases according to $k$ from $0$ to $m_0 = \lfloor \frac{n-2}{3} \rfloor$ or $\lceil \frac{n-2}{3} \rceil - 1$ and then it monotone decreases according to $k > m_0$. At $k = m_0$, $P(X=k)$ reaches the maximum probability. It indicates that the case that there are  $m_0$ frequency $K_4$s where $(A,B)$ has the frequency 3 and $n - 3 - m_0$ frequency $K_4$s where $(A,B)$ has the frequency 5 has the maximum probability. In this case, the average frequency $f(A,B) = \frac{13}{3}$ and $p_4(e\in OP^4) = \frac{13}{18}$ exist based on the $n-3$ frequency $K_4$s. In the worst case, $(A,B)$ has the frequency 3 in each of the $n-3$ frequency $K_4$s. The average frequency $f(A,B) = 3$ and the probability $p_4(e\in OP^4) = \frac{1}{2}$ hold based on the $n-3$ frequency $K_4$s. However, based on the binomial distribution (\ref{eq1}), the probability is $P(X=n-3) = \left(\frac{1}{3}\right)^{n-3}$ which approaches zero as $n$ is big enough. 

\begin{equation}
	P(X = k) = {{n-3}\choose{k}}(p_3(e))^k(1 - p_3(e))^{n-3-k} = {{n-3}\choose{k}}\left(\frac{1}{3}\right)^k\left(\frac{2}{3}\right)^{n-3-k}
	\label{eq1}
\end{equation} 
	
	Besides the $n-3$ $ABCD$s containing $(A,B)$, there are the other ${{n-3}\choose{2}}$ $ABCD$s containing $(A,B)$ in $K_n$. Let $N_p = {{n-2}\choose{2}}$, $N_q = {{n-3}\choose{2}} $, and $N_r = n-3$, and let $p_{\{3,5\}}(e) = p_3(e) + p_5(e)$, $q_{\{3,5\}}(e) = p_3(e) + p_5(e)$ and $r_{\{3,5\}}(e) = p_3(e) + p_5(e)$ denote the probability that $(A,B)$ has the frequency 3 and 5 with respect to the $N_p$, $N_q$ and $N_r$ frequency $K_4$s, respectively. According to the $n-3$ vertex-disjoint edges $(C,D)$ inside the $OHC$ for $(A,B)$, $(A,B)$ and $(C,D)$ are contained in the $N_r = n-3$ frequency $ABCD$s. Moreover, $(A,B)$ has the frequency 3 or 5 in each of the frequency $ABCD$s, respectively. Thus, $r_{\{3,5\}}(e) = 1$ exists based on the $N_r$ frequency $ABCD$s. In the $N_p$, $N_q$ and $N_r$ frequency $K_4$s, the number of frequency $K_4$s where $(A,B)$ has the frequency 3 and 5 are $M_p = N_pp_{\{3,5\}}(e)$, $M_q = N_qq_{\{3,5\}}(e)$, and $M_r = N_rr_{\{3,5\}}(e)$, respectively. Since $M_p = M_q + M_r$, the formula (\ref{eq2}) is derived. Because $r_{\{3,5\}}(e) = 1$, $p_{\{3,5\}} \neq 0$ holds. As $n$ is sufficiently big, $\frac{2}{n-2}$ tends to zero, and $p_{\{3,5\}}(e) =  q_{\{3,5\}}(e)$ will hold for $r_{\{3,5\}}(e) \leq 1$. As $p_{\{3,5\}}(e) =  q_{\{3,5\}}(e)$ is considered, $p_{\{3,5\}}(e) =  q_{\{3,5\}}(e) = r_{\{3,5\}} = 1$ is derived. It indicates that $(A,B)$ has the frequency 3 or 5 in each frequency $ABCD$ in $K_n$ on average. 
	
	\begin{equation}
		p_{\{3,5\}}(e) =   \left[1 - \frac{2}{n-2}\right]\times q_{\{3,5\}}(e) + \frac{2}{n-2}\times r_{\{3,5\}}(e)
		\label{eq2}
	\end{equation}
	
	As an edge has the frequency 3 or 5 in a frequency $K_4$, it is the $OHC$ edge of the $K_4$. Thus, $(A,B)$ is the $OHC$ edge of any $K_4$ containing it. The optimality of $(A,B)$ is preserved to all the sub-graphs $K_4$ containing it. As the probability $p_5(e)$, $q_5(e)$ and $r_5(e)$ that $(A,B)$ has the frequency 5 are considered with respect to the $N_p$, $N_q$ and $N_r$ frequency $K_4$s, $p_5(e) = q_5(e) = r_5(e)$ can be proven according to the same proof procedure. In the average case, $p_5(e) = q_5(e) = r_5(e) = \frac{2}{3}$ holds for $(A,B)$. Thus, the average frequency $f(A,B) = \frac{13}{3}$ is computed based on all the frequency $K_4$s containing $(A,B)$, and the probability $p(e\in OP^4) = \frac{13}{18}$. In the worst case, $(A,B)$ has the frequency 3 in each of the $N_p$ frequency $K_4$s. The average frequency $f(A,B) = 3$ and the probability $p_4(e\in OP^4) = \frac{1}{2}$. However, the probability $P(X = N_p) = (\frac{1}{3})^{N_p}$ tends to zero as $n$ is big enough based on the binomial distribution (\ref{eq1}). 
\end{proof}

\begin{figure}
	\centering
	\includegraphics[width=2in,bb=0 0 300 100]{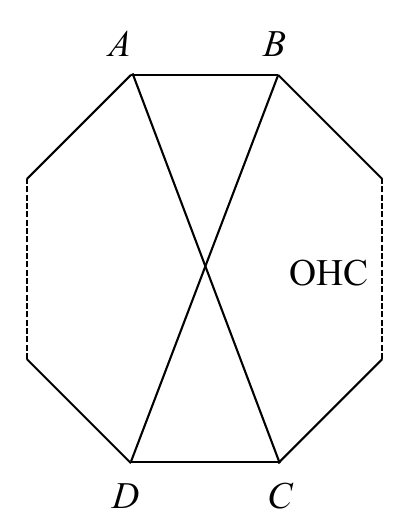}
	\caption{An edge $(A,B)\in OHC$ in $K_n$ and the $n-3$ vertex-disjoint edges $(C,D)\in OHC$.}
	\label{OHCKn}
\end{figure}
	
	Theorem \ref{th001} indicates that an $OHC$ edge in $K_n$ is also one $OHC$ edge in a $K_4$ containing it based on the average frequency and probability computed with frequency $K_4$s. For an ordinary edge $g=(A,C)$ in $K_n$, it is contained in $n-3$ $ABCD$s with an adjacent $OHC$ edge $e=(A,B)$. In each of these corresponding frequency $K_4$s, $(A,B)$ has the frequency 3 or 5. Based on the bijection between the frequency of $(A,B)$ in a frequency $K_4$ and the distance sum in the corresponding $K_4$, $S_1 < S_2$ or $S_1 < S_3$ holds in the $K_4$. If $S_1 < S_2$, there are three frequency $K_4$s (a), (b) and (e) in Figure \ref{quadrilaterals}. If $S_1 < S_3$, the three frequency $K_4$s are (a), (b) and (c) in Figure \ref{quadrilaterals}. If $S_1 < S_2$ and $S_1 < S_3$ happen with the equal probability, $(A,C)$ has the frequency 1 three times, 3 twice, and 5 once based on the six frequency $K_4$s. Thus, $p_1(e) = \frac{1}{2}$, $p_3(e) = \frac{1}{3}$ and $p_5(e) = \frac{1}{6}$ holds for $(A,C)$ based on the six frequency $K_4$s. $(A,C)$ has four adjacent $OHC$ edges $(A,B)$ in $K_n$. Thus, there are $4(n-4)$ frequency $K_4$s where the probability $p_1(e) = \frac{1}{2}$, $p_3(e) = \frac{1}{3}$ and $p_5(e) = \frac{1}{6}$ holds for $(A,C)$. One can use the similar proof procedure used in Theorem \ref{th001} to prove that $p_1(e) = \frac{1}{2}$, $p_3(e) = \frac{1}{3}$ and $p_5(e) = \frac{1}{6}$ holds for $(A,C)$ based on all frequency $K_4$s containing $(A,C)$. In this case, the average frequency $f(A,C) = \frac{7}{3}$ which is smaller than the average frequency 3 of all edges in $K_n$, and the probability $p(e\in OP^4) = \frac{7}{18} < \frac{1}{2}$ exists in the average case. The average frequency and probability for an ordinary edge in the average case is still smaller than that of an $OHC$ edge in the worst case, respectively.
	
	In the next, we shall study the frequency of an edge $e\in OHC$ in $K_n$ based on frequency $K_i$s where $i\in[4,n]$. As $i \geq 5$, $p_i(e\in OP^i)$ is noted as $p_i(e)$ for convenience. If $i$ is big, the number of frequency $K_i$s becomes large. It is time-consuming to derive all the frequency $K_i$s for all kinds of $K_i$s with various distances on edges. Here, the frequency of an $OHC$ edge in $K_i$ is analyzed based on the structure relationships among edges, $OP^i$s and $OHC$. Firstly, the average frequency of $e$ in a frequency $K_i$ is studied for the average case. 

\begin{theorem}
	\label{th01}
	Given a $K_i$ containing one $OHC$ and ${{i}\choose{2}}$ $OP^i$s where $i\in[4,n]$, an $OHC$ edge is contained in more than $\frac{2(i-1)^2}{5}$ $OP^i$s, whereas an ordinary edge is contained in at most $\frac{(i-1)^2}{5(i-3)}$ $OP^i$s in the average case. 
\end{theorem}

\begin{proof}
	Given a $K_i$ where $i\in[4,n]$, it contains ${i}\choose{2}$ $OP^i$s and one $OHC$. The $i-1$ edges containing a vertex $v$ are considered. The $i-1$ edges include two $OHC$ edges and $i-3$ ordinary edges. Firstly, we assume that the $i-1$ edges are uniformly contained in the $OP^i$s. Based on the ${{i}\choose{2}}$ $OP^i$s, the total frequency of the $i-1$ edges containing $v$ is $(i-1)^2$. Thus, each edge is contained in $i-1$ $OP^i$s on average. For the two $OHC$ edges containing $v$, each of them is contained in $i-1$ $OP^i$s in $OHC$, respectively. Besides the $i$ $OP^i$s inside the $OHC$, there are the other $\frac{i(i-3)}{2}$ $OP^i$s outside the $OHC$. Since the two $OHC$ edges are contained in the $i-1$ $OP^i$s inside the $OHC$, they will never be contained in any one of the other $\frac{i(i-3)}{2}$ $OP^i$s. 
	
	Given an $OP^i$, $v$ is either one endpoint or an intermediate vertex. As it is the endpoint, the $OP^i$ contains one edge containing $v$. Otherwise, the $OP^i$ includes one pair of edges containing $v$. Moreover, $v$ is the endpoint of $i-1$ $OP^i$s, and it is the intermediate vertex of ${{i-1}\choose{2}}$ $OP^i$s. According to the $OHC$, $v$ is the endpoint of two $OP^i$s and the intermediate vertex of $i-2$ $OP^i$s inside the $OHC$. Thus, $v$ is the endpoint of $i-3$ $OP^i$s and intermediate vertex of ${{i-2}\choose{2}}$ $OP^i$s outside the $OHC$. Since the $i-3$ ordinary edges are uniformly distributed in the $\frac{i(i-3)}{2}$ $OP^i$s outside the $OHC$, they have the same property to build these $OP^i$s whether $v$ is one endpoint or intermediate vertex. Thus, each pair of ordinary edges is contained in the same number of $OP^i$s in which $v$ is one intermediate vertex, respectively. There are ${{i-3}\choose{2}}$ pairs of ordinary edges containing $v$. If each pair of the ordinary edges is contained in more than one $OP^i$ where $v$ is one intermediate vertex, the total number of $OP^i$s will be bigger than ${{i}\choose{2}}$. Thus, each pair of the ordinary edges is contained in at most one such $OP^i$. As each pair of the ordinary edges is contained in one $OP^i$ in which $v$ is the intermediate vertex, there are total ${{i-3}\choose{2}}$ such $OP^i$s each of which includes two ordinary edges containing $v$.
	Besides these $OP^i$s, there are $i-3$ more $OP^i$s which may contain the $i-3$ ordinary edges. In each of the $i-3$ $OP^i$s, $v$ is one endpoint, and each of the $OP^i$s includes one ordinary edge containing $v$. Thus, the $i-3$ ordinary edges are contained in at most ${{i-2}\choose{2}}$ $OP^i$s without including the two $OHC$ edges. These $OP^i$s visit the $i-3$ ordinary edges $(i-3)^2$ times. Thus, each ordinary edge containing $v$ is contained in $i-3$ $OP^i$s rather than $i-1$ $OP^i$s on average. It indicates that each ordinary edge can not be contained in $i-1$ $OP^i$s due to the lack of the two $OHC$ edges containing $v$. 
	
	In addition, the ${{i}\choose{2}}$ $OP^i$s visiting $v$ can not be fully built if the two $OHC$ edges are not contained in some $OP^i$s outside the $OHC$. Besides the $i$ $OP^i$s in $OHC$ and  ${{i-2}\choose{2}}$ $OP^i$s only including the ordinary edges, the remainder $i-3$ $OP^i$s must include the two $OHC$ edges. Moreover, $v$ is one intermediate vertex in each of the $i-3$ $OP^i$s. On average, if the two $OHC$ edges are uniformly contained in the $i-3$ $OP^i$s while the $i-3$ ordinary edges also do, each $OHC$ edge is included in at least $\frac{3i-5}{2}$ $OP^i$s in the worst average case, whereas each ordinary edge is contained in at most $i-2$ $OP^i$s. It implies that an ordinary edge will be contained in smaller than $i-1$ $OP^i$s in the best average case. 
	
	In the worst average case, an $OHC$ edge is contained in the $OP^i$s with the bigger probability than an ordinary edge.  Because each ordinary edge is impossibly contained in $OHC$, we assume that the two $OHC$ edges and $i-3$ ordinary edges containing $v$ are uniformly contained in the $\frac{i(i-3)}{2}$ $OP^i$s excluding from $OHC$. In this case, an $OHC$ edge is contained in $2i - 4$ $OP^i$s in the worst average case, whereas an ordinary edge is contained in at most $i - 3$ $OP^i$s in the best average case. 
	
	In fact, the $i-3$ ordinary edges are not contained in the $\frac{i(i-3)}{2}$ $OP^i$s with the equal probability. Otherwise, the number of $OP^i$s will be bigger than ${{i}\choose{2}}$ if $i > 4$. If $i > 4$, most ordinary edges are contained in smaller than $i-3$ $OP^i$s, and some ordinary edges are contained in none of the $OP^i$s. It mentions that the total frequency of the $i-1$ edges containing $v$ is $(i-1)^2$ based on the ${{i}\choose{2}}$ $OP^i$s. Since the ordinary edges with a frequency smaller than $i-3$ cannot be used to construct the $OP^i$s at the assumed positions, they must be replaced by the $OHC$ edges and some other ordinary edges. In this case, the total frequency of the two $OHC$ edges will increase while that of the $i-3$ ordinary edges will decrease. Thus, an $OHC$ edge will be contained in greater than $2i-4$ $OP^i$s on average, whereas an ordinary edge will be contained in smaller than $i-3$ $OP^i$s in the average case. 
	
	There are $i$ $OHC$ edges and $\frac{i(i-3)}{2}$ ordinary edges in the $K_i$. Given an $OHC$ edge $e$ and ordinary edge $g$ containing $v$ in the  $K_i$, the $OP^i$s containing $e$ will be more than two times of those $OP^i$s containing $g$ because $\frac{2(i-2)}{i-3} > 2$ holds. It means that the frequency of $e$ is bigger than two times of that of $g$ in the frequency $K_i$, i.e., $f(e) > 2f(g)$. As the probability that $e$ or $g$ is contained in an $OP^i$ is considered, $p_i(e) > 2p_i(g)$ holds for $e$ and $g$ where $p_i(e)$ and $p_i(g)$ denote the probability that $e$ and $g$ are contained in an $OP^i$, respectively. 
	This result holds for an $OHC$ edge and an adjacent ordinary edge in the worst average case. 
	
	In the $K_i$, each edge is contained in $i-2$ $K_{i-1}$s each of which contains ${{i-1}\choose{2}}$ $OP^{i-1}$s. Each edge in the $K_i$ is contained in the specific number of $OP^{i-1}$s, respectively. Moreover, the $OP^i$s only contain the $OP^{i-1}$s in the $K_i$. Due to the variations of $TSP$, each $OP^{i-1}$ containing $v$ is contained in the $OP^i$s with the equal probability. Since $p_i(e) > 2p_i(g)$ exists based on the frequency $K_i$, $p_{i-1}(e) > 2p_{i-1}(g)$ also holds, i.e., the number of $OP^{i-1}$s containing $e$ will be more than two times of those containing $g$. If $g$ appears once in one $OP^{i-1}$, $e$ will appear more than twice in the $OP^{i-1}$s. In the same manner, $p_k(e) > 2p_k(g)$ holds where $(k\in[4,i])$, and the number of $OP^k$s including $e$ will be more than two times of those containing $g$ until $k = 4$. As $e$ and $g$ are contained in the $K_4$s, the number of $OP^4$s containing $e$ will be more than two times of those containing $g$. Because $p_4(e) > 2p_4(g)$ exists, $f(e) > 2f(g)$ also holds in the frequency $K_4$s containing $e$ and $g$. In a frequency $K_4$, the frequency of an edge is 1, 3 or 5, see the six frequency $K_4$s in Figure \ref{quadrilaterals}. Based on Theorem \ref{th001}, $e$ is the $OHC$ edge in any $K_4$ containing it. Thus, $f(e)$ takes the value 3 or 5 in these frequency $K_4$s. Because $f(e) > 2f(g)$, $f(g)$ must take the frequency 1 in the frequency $K_4$s containing $e$ and $g$. In this case, $e$ is the $OHC$ edge of these $K_4$s whereas $g$ is an ordinary edge. Since $g$ is arbitrary, it indicates that $g$ will be an ordinary edge in a $K_4$ where there is an $OHC$ edge in $K_i$. 
	
	Not all the $OP^k$s in each $K_k$ are contained in the $OP^i$s. In fact, only part of the $OP^k$s in a small number of $K_k$s are contained in the $OP^i$s. For example, the $OHC$ in the $K_i$ contains $i$ $OP^4$s, and each of the $OP^4$s is contained in one $K_4$ where there are six $OP^4$s, respectively. In the extreme case, all the $OP^4$s contained in the ${{i}\choose{2}}$ $OP^i$s are different. There are at most $\frac{i(i^2-6i+11)}{2}$ $OP^4$s. However, the $K_i$ includes $6{{i}\choose{4}}$ $OP^4$s. It indicates that some or most $OP^4$s are not contained in the $OP^i$s as $i \geq 5$. Moreover, some $OP^4$s are contained in a number of the $OP^i$s. Thus, the number of different $OP^4$s contained in the $OP^i$s will be much smaller than $\frac{i(i^2-6i+11)}{2}$. To approve $p_4(e) > 2p_4(g)$ and $p_i(e) > 2p_i(g)$, the $OP^4$s in the $K_4$s containing $e$ will be used to construct the $OP^i$s. Moreover, $e$ is the $OHC$ edge in each of these $K_4$s and it is contained in at least three $OP^4$s, respectively. If there is one ordinary edge $g$ in such a $K_4$, it is contained in only one $OP^4$.

	In most of the $K_4$s without  containing $e$, $f(e) = 0$ exists in the corresponding frequency $K_4$s. It does not meet the condition $p_4(e) > 2p_4(g)$. Thus, the $OP^4$s in these $K_4$s without including $e$ are seldom used to build the $OP^i$s. It means that the $OP^4$s containing more $OHC$ edges and less ordinary edges are used to build the $OP^i$s. The $OP^4$s in the $OP^i$s will contain the $OHC$ edges as most as possible due to  $p_i(e) > 2p_i(g)$.  Since an $OP^i$ only contains the $OP^4$s, each $OP^i$s will contain the $OHC$ edges as most as possible. Meanwhile, the $OP^k$s composed of the $OP^4$s also contain the $OHC$ edges as most as possible. Every pair of $e$ and $g$ is contained in $i-3$ $K_4$s in the $K_i$. Given such one $K_4$, if the $OP^4$ including $g$ in the $K_4$ is contained in one $OP^i$, the $OP^4$ and other $OP^4$s containing $e$ in the $K_4$ will be contained in more than two $OP^i$s in the average case because of $p(e\in OP^i) > 2p(g\in OP^i)$. In fact, as the $OP^4$s in the $K_4$ are contained in the $OP^i$s with the equal probability, there will be more than three $OP^i$s containing $e$ because there are at least three $OP^4$s containing $e$ in the $K_4$ where there is only one $OP^4$ contains $g$. 
	
	
	It mentions that there are two $OHC$ edges $e$ and $i-3$ ordinary edges $g$ containing $v$. If one $g$ is contained in an $OP^i$, each of the two $e$s will be contained in more than two $OP^i$s. Moreover, it is the $OP^4$s containing $e$ and $g$ in the corresponding $K_4$s that are also contained in the $OP^i$s. 
	Given two $g$s, they and $e$ are contained in two $K_4$s, respectively, because $f(g) = 1$ appears once in a frequency $K_4$. The $OP^4$s in the two $K_4$s are different from each other. As the $OP^4$s including each of the two $g$s are contained in the $OP^i$s,  the $OP^4$s including $e$ in the two $K_4$s are contained in the $OP^i$s with respect to the two $g$s, respectively. Thus, $p_i(e) > 2p_i(g)$ occurs independently for every pair of $e$ and $g$ contained in one $K_4$. It indicates that $f(e)$ will be bigger than two times of that of all adjacent ordinary edges $g$ contained in the $OP^i$s, i.e. $f(e) > 2\sum_{k=1}^{i-3}f_k(g)$ where $f_k(g)$ is the frequency of the $k^{th}$ ordinary edge $g$. 
	Since there are two $e$s containing $v$, the total frequency of the two $e$s containing $v$ is more than four times of that of the $i-3$ ordinary edges $g$. 
	
	Given another ordinary edge $g_1$ containing $v$, $g$ will replace $g_1$ for building the $OP^i$s if $g$ has the frequency $f(g) > i-3$ whereas $g_1$ has the smaller frequency $f(g_1) < i-3$ in the frequency $K_i$. As $g$ replaces $g_1$ for building the $OP^i$s, it is the $OP^{i-1}$s containing $g$ that replace the $OP^{i-1}$s containing $g_1$ in the $OP^i$s. 
	Since $p_{i-1}(e) > 2p_{i-1}(g)$ exists, the $OP^{i-1}$s containing $e$ will also replace those containing $g_1$ for building the $OP^i$s. Moreover, more than two times of the $OP^{i-1}$s containing $e$ will replace the $OP^{i-1}$s containing $g_1$ to build the $OP^i$s. Thus, $f(e) > 2f(g)$ still holds in the frequency $K_i$, and $p(e\in OP^i) > 2p(g\in OP^i)$ always holds no matter how $f(g)$ grows. As most ordinary edges are largely replaced by the $OHC$ edges and some other ordinary edges in the $OP^i$s, the frequencies of most ordinary edges will be very small, and the $OHC$ edges and some ordinary edges will have the big frequencies. Whatever the frequencies of the ordinary edges change, the total frequency is smaller than half frequency of each of the adjacent $OHC$ edges on average. 
	
	The total frequency of the $i-1$ edges containing $v$ is $(i-1)^2$. The total frequency of the two $OHC$ edges is at least $\frac{4(i-1)^2}{5}$, and each $OHC$ edge has the average frequency bigger than $f(e) = \frac{2(i-1)^2}{5}$. The total frequency of the $i-3$ ordinary edges is smaller than $\frac{(i-1)^2}{5}$, and the average frequency of an ordinary edge $g$ will be smaller than $f(g) = \frac{(i-1)^2}{5(n-3)}$. If there is only one ordinary edge $g$ with $f(g) > 0$, $g$ has the upper frequency bound $f(g) = \frac{(i-1)^2}{5}$. 
\end{proof}	

Theorem \ref{th01} implies that an $OHC$ edge of a $K_i$ is contained in more  $OP^i$s than an ordinary edge. In a $K_i$, the $OP^i$s contain the $OP^k$s ($k\in[4,i]$) having a smaller number of vertices. Since the edges are not uniformly distributed in the $OP^k$s, the $OHC$ edges and ordinary edges have different frequencies based on the $OP^k$s. On average, an $OHC$ edge $e$ is contained in the $OP^k$s with a bigger probability than an ordinary edge $g$ for $p_k(e) > 2p_k(g)$. Based on Theorem \ref{th01}, the $OP^k$s including more $OHC$ edges will be contained in the $OP^i$s with the bigger probability. Moreover, most $OP^k$s in the $OP^i$s are contained in the $K_k$s including the $OHC$ edges in the $K_i$. On the other hand, the $OP^k$s including more ordinary edges will be seldom contained in the $OP^i$s. %

After the frequency $K_i$ is computed with the ${{i}\choose{2}}$ $OP^i$s, the average frequency of all edges is $i-1$. The lower frequency bound for $OHC$ edges is concerned to separate $OHC$ edges from ordinary edges. Theorem \ref{th01} has given one lower frequency bound $2(i-2)$  for $OHC$ edges. However, this frequency bound is derived under the assumption that the ordinary edges are uniformly distributed in the $OP^i$s, and it is too small comparing to the upper frequency bound $\frac{(i-1)^2}{5}$ for ordinary edges. 
In the next, the bigger lower frequency bound for $OHC$ edges in a $K_i$ will be proven for the worst average case. Moreover, the change of the probability $p_k(e)$ for an $OHC$ edge in the $K_i$ will be analyzed in detail according to $k\in[4,i]$. 

\begin{theorem}
	\label{th002}
	 Given a $K_{i}$ containing one $OHC$ and ${{i}\choose{2}}$ $OP^i$s, the lower frequency bound of the $OHC$ edges $e$ is $\frac{1}{2}{{i}\choose{2}}$, and the probability $p_i(e) \geq \frac{1}{2}$ where $i\in [4,n]$. 
\end{theorem}

\begin{proof}
	
	For convenience of the proof, let $i:=i - 1$. Given a $K_{i+1}$ where $i\in [3, n-1]$, it contains $i+1$ $OHC$ edges and $i+1$ $K_i$s. Given an $OHC$ edge $e=(A,B)$ in the $K_{i+1}$, let $p_i(e)$ and $p_{i+1}(e)$ denote the probability that $e$ is contained in the optimal $i$-vertex paths and optimal $(i+1)$-vertex paths in the $K_i$s and $K_{i+1}$ containing $e$, respectively. The relationship between $p_{i+1}(e)$ and $p_i(e)$ will be analyzed to prove the lower frequency bound $\frac{1}{2}{{i+1}\choose{2}}$ according to the $OP^{i+1}$s in the $K_{i+1}$. 
	If $i=3$, $p_{i+1}(e)=\frac{1}{2}$ and $\frac{5}{6}$ based on the six frequency $K_4$s in Figure \ref{quadrilaterals}, i.e., $p_4(e)\geq 1/2$. 
	
	There are ${{i+1}\choose{2}}$ optimal $(i+1)$-vertex paths and one $OHC$ in the $K_{i+1}$. The $OHC$ contains $i+1$ optimal $(i+1)$-vertex paths and $i$-vertex paths, respectively. Comparing to the optimal $i$-vertex paths excluding from $OHC$ of the $K_{i+1}$, the optimal $i$-vertex paths in the $OHC$ are contained in the optimal $(i+1)$-vertex paths (in the $OHC$) with the biggest probability. In other words, each optimal $i$-vertex path in the $OHC$ is contained in two optimal $(i+1)$-vertex paths whereas an optimal $i$-vertex path outside the $OHC$ is contained in at most one optimal $(i+1)$-vertex path. Since $e$ appears in the optimal $i$-vertex paths and $(i+1)$-vertex paths lying both inside and outside the $OHC$, we also investigate the probability that $e$ belongs to optimal $i$-vertex paths and $(i+1)$-vertex paths outside the $OHC$. Let $q_i(e)$ and $q_{i+1}(e)$ denote the probability that $e$ is contained in the optimal $i$-vertex paths and optimal $(i+1)$-vertex paths excluding from the $OHC$, respectively. 
	
	It mentions that each optimal $(i+1)$-vertex path has one special pair of endpoints in the $K_{i+1}$. Due to the various distances of the edges in the $K_{i+1}$, any optimal $i$-vertex path outside the $OHC$ in the $K_{i+1}$ may be included in the optimal $(i+1)$-vertex paths outside the $OHC$. Thus, we assume that the optimal $i$-vertex paths outside the $OHC$ in the $K_{i+1}$ are included in the optimal $(i+1)$-vertex paths excluding from the $OHC$ with the equal probability. Moreover, each optimal $i$-vertex path is included in the optimal $(i+1)$-vertex paths in the $OHC$ of the $K_{i+1}$ with another equal probability. 
	
	Since $e$ is in the $OHC$ of the $K_{i+1}$, it is contained in $i$ optimal $(i+1)$-vertex paths in the $OHC$. In addition, $e$ is contained in $i-1$ $K_i$s each of which contains ${{i}\choose{2}}$ optimal $i$-vertex paths,  respectively. There are $i+1$ $K_i$s in the $K_{i+1}$. The $OHC$ of $K_{i+1}$ contains $i+1$ vertices. Delete one vertex from the $OHC$, an optimal $i$-vertex path is obtained. Moreover, the remaining $i$ vertices in the optimal $i$-vertex path are contained in one corresponding $K_i$. Thus, each $K_i$ in the $K_{i+1}$ contains one optimal $i$-vertex path in the $OHC$, and there are $i+1$ such optimal $i$-vertex paths. Moreover, each of the optimal $i$-vertex paths is contained in two optimal $(i+1)$-vertex paths in the $OHC$, respectively, because the $OHC$ is one closed cycle. 
	
	Besides the $i+1$ optimal $i$-vertex paths in the $OHC$, there are $(i+1)(i-2)$ optimal $i$-vertex
	paths contained in the other $\frac{(i+1)(i-2)}{2}$ optimal $(i+1)$-vertex paths in the $K_{i+1}$. It is known that each of the optimal $(i+1)$-vertex paths has one specified pair of endpoints, respectively. Moreover, each of the optimal $(i+1)$-vertex path contains only two optimal $i$-vertex paths in the $K_{i+1}$,  respectively. Given such an optimal $(i+1)$-vertex path, delete one endpoint from the optimal $(i+1)$-vertex path, an optimal $i$-vertex path will be obtained. The $i$ vertices in the optimal $i$-vertex path are also contained in one corresponding $K_i$ in the $K_{i+1}$. Moreover, the two optimal $i$-vertex paths in the optimal $(i+1)$-vertex path excluding from the $OHC$ are contained in two related $K_i$s in the $K_{i+1}$. The two related $K_i$s have two different vertices which are taken as the endpoints of the optimal
	$(i+1)$-vertex path. It says that the two optimal $i$-vertex paths in the two related $K_i$s have one common optimal $(i-1)$-vertex path in the $K_{i+1}$. Thus, any pair of related $K_i$s having a pair of different vertices contain two optimal $OP^i$s which are used to construct one $OP^{i+1}$ in the $K_{i+1}$. 
	
	As one vertex $v$ is taken as the endpoint of the optimal $(i+1)$-vertex paths, there are $i$ such optimal $(i+1)$-vertex paths. In the $OHC$, there are two optimal $i+1$-vertex paths where $v$ is one endpoint. The residual $i-2$ optimal $(i+1)$-vertex paths where $v$ is one endpoint will contain $i-2$ optimal $i$-vertex paths without $v$. Thus, the residual $i-2$ optimal
	$i$-vertex paths without $v$ are contained in the same $K_i$ which is generated from the $K_{i+1}$ by deleting $v$. Because there are $i+1$ pairs of $K_i$ and $v$ in the $K_{i+1}$, each of the $K_i$s includes $i-2$ optimal $i$-vertex paths contained in the optimal $(i+1)$-vertex paths excluding 	from the $OHC$, respectively. Plus the optimal $i$-vertex path in the $OHC$, each $K_i$ includes $i-1$ optimal $i$-vertex paths which are contained in the optimal $(i+1)$-vertex paths. 
	
	Each $K_i$ includes ${{i}\choose{2}}$ optimal $i$-vertex paths. Only $i-1$ of them are contained in the optimal $(i+1)$-vertex paths in the  $K_{i+1}$. In addition, each optimal $(i+1)$-vertex path is composed of two optimal $i$-vertex paths. Thus, it just requires half of the optimal $i$-vertex paths to compute all the optimal $(i+1)$-vertex paths. Given an optimal $i$-vertex path, the probability that it is contained in the optimal $(i+1)$-vertex paths is $\frac{\frac{i-1}{2}}{{{i}\choose{2}}} = \frac{1}{i}$ on average if the optimal $i$-vertex paths inside and outside the $OHC$ are not separately discussed.  Given an optimal $i$-vertex path, it is contained in the optimal $(i+1)$-vertex paths or not. Moreover, as it is contained in an optimal $(i+1)$-vertex path, the optimal $(i+1)$-vertex will be either in the $OHC$ of the $K_{i+1}$ or not. 
	
	$e=(A,B)$ is contained in $i-1$ $K_i$s where there are $i-1$ optimal $i$-vertex paths in the $OHC$ of the $K_{i+1}$, $(i-1)(i-2)$ optimal $i$-vertex paths in the other optimal $(i+1)$-vertex paths excluding from the $OHC$, and total $(i-1){{i}\choose{2}}$ optimal $i$-vertex paths. For the $K_i$ without vertex $A\in e$ and the other $K_i$ without vertex $B\in e$, they include two optimal $i$-vertex paths contained in the optimal $(i+1)$-vertex path where $A$ and $B$ are the endpoints. Because $e=(A,B)$ is in the $OHC$ and there is only one optimal $(i+1)$-vertex path with the endpoints $A$ and $B$ in the $K_{i+1}$, this optimal $(i+1)$-vertex path with endpoints $A$ and $B$ must be in the $OHC$. In addition, $B$ is one endpoint of two optimal $(i+1)$-vertex paths in the $OHC$. The other optimal $(i+1)$-vertex path is generated by the same optimal $i$-vertex path in the $K_i$ without $B$ through connecting $B$ to $A$ in case that $A$ becomes one intermediate vertex in the optimal $(i+1)$-vertex path. Note $C$ as the other endpoint of this optimal $i$-vertex path in the $K_i$ and it is noted as $OP^i=(A,...,C)$. The $K_i$ without $B$ also contains the $i-2$ optimal $i$-vertex paths which are inside the optimal $(i+1)$-vertex paths excluding from the $OHC$. Moreover, $B$ is one endpoint of these $i-2$ optimal $(i+1)$-vertex paths because an optimal $(i+1)$-vertex path is generated by one optimal $i$-vertex path adding the remaining vertex at one specified end. 
	
	As $B$ is one endpoint of the $i-2$ optimal $(i+1)$-vertex paths, the other endpoints are the vertices in the $K_{i+1}$ except $A$,  respectively. Delete $B$ from these optimal $(i+1)$-vertex paths, we will obtain $i-2$ optimal $i$-vertex paths in the $K_i$ without $B$. Each of the $i-2$ optimal $i$-vertex paths contains at least one ordinary edge. For avoid becoming shorter through edges replacement, an optimal $(i+1)$-vertex path will contain the biggest number
	of $OHC$ edges and the least number of ordinary edges. Thus, the optimal $i$-vertex paths where $A$ or $C$ is one endpoint will be preferentially used to construct the optimal $(i+1)$-vertex paths since one more $OHC$ edge $(A,B)$ or $(B,C)$ will be added to each of them, respectively. In the $K_i$ without $A$, $i-2$ optimal $i$-vertex paths are also selected to build the optimal $(i+1)$-vertex paths excluding from the $OHC$. If $A$ is adjacent to the other vertex $D$ in the $OHC$ of $K_{i+1}$, the optimal $i$-vertex paths with one endpoint $B$ and $D$ will be preferentially selected to construct the optimal $(i+1)$-vertex paths. 
	
	Except the two optimal $i$-vertex paths in the $OHC$ contained in the pair of $K_i$s without $A$ and $B$, respectively, the $i-1$ $K_i$s containing $e = (A,B)$  include all the optimal $i$-vertex paths which are contained in the optimal $(i+1)$-vertex paths excluding from the $OHC$, and the $i-1$ optimal $i$-vertex paths in the $OHC$. Since an optimal  $(i+1)$-vertex path is constructed by an optimal $i$-vertex path and one more vertex, it just requires the $\frac{(i+1)(i-2)}{2}$ optimal $i$-vertex paths in the $i-1$ $K_i$s to build the optimal $(i+1)$-vertex paths in the $K_{i+1}$. Moreover, any optimal $i$-vertex path excluding from the $OHC$ will be contained in the optimal $(i+1)$-vertex paths excluding from the $OHC$ with the equal probability. Thus, the probability
	that any given optimal $i$-vertex path outside the $OHC$ is one of the $\frac{(i+1)(i-2)}{2}$ 
	optimal $i$-vertex paths is computed as $q = \frac{(i+1)(i-2)/2}{(i-1){{i}\choose{2}} - (i-1)} =\frac{1}{i-1}$. 
	
	Based on the probability $q_i(e)$, the number of optimal $i$-vertex paths outside the $OHC$ while containing $e$ is computed as $M_i = (i-1)\left({{i}\choose{2}} - 1\right)q_i(e)$. Each of the $M_i$ optimal $i$-vertex paths will be contained in the optimal $(i+1)$-vertex paths excluding from the $OHC$. After the $i+1$ optimal $i$-vertex paths in the $OHC$ are subtracted, the probability that $e$ is contained in the optimal $(i+1)$-vertex paths excluding from the $OHC$ is computed as formula (\ref{eq3}). Since $q_{i+1}(e)=q_i(e)$, it indicates that the probability that $e$ is contained in the optimal $i$-vertex paths excluding from the $OHC$ remains unchanged from $i$ to $i+1$ where $i\in [4,n-1]$. As an ordinary edge in the $K_{i+1}$, $g = (A,C)$ also preserves the same  probability $q_i(e)$ from $i$ to $i+1$. 
	
	\begin{equation}
		q_{i+1} = \frac{M_iq}{{{i+1}\choose{2}}-(i+1)} = \frac{k-1}{k-1}\times q_i = q_i
		\label{eq3}
	\end{equation}
	
	In the following, the $i-1$ optimal $i$-vertex paths in the $OHC$ will be considered to predict the change of $p_i(e)$ and $p_i(g)$ from $i$ to $i+1$. Given an optimal $i$-vertex path in the $OHC$, the probability that it is contained in the optimal $(i+1)$-vertex paths is at least two times of that for an optimal $i$-vertex path outside the $OHC$. Let $r$ denotes the probability that an optimal $i$-vertex path in the $OHC$ is contained in the optimal $(i+1)$-vertex paths. $r \geq 2q = \frac{2}{i-1}$ holds. It means that an optimal $i$-vertex path in the $OHC$ is used to build more than two times of the optimal $(i+1)$-vertex paths comparing to an optimal $i$-vertex path outside the $OHC$. For an optimal $i$-vertex path contained in one $OP^{i+1}$ outside the $OHC$, $r = 2q$ holds. Otherwise, $r > 2q$ exists. 
	
	As an optimal $i$-vertex path is contained in an optimal $(i+1)$-vertex path, this optimal $(i+1)$-vertex path may be in the $OHC$ or not, and it has different probabilities to be inside and outside the $OHC$. In the $K_{i+1}$, there are $i+1$ optimal $i$-vertex paths in the $OHC$, and $(i+1)(i-2)$ optimal $i$-vertex paths in the optimal $(i+1)$-vertex paths outside the $OHC$. It requires half of the $(i+1)(i-2)$ optimal $i$-vertex paths to build all the optimal $(i+1)$-vertex paths outside the $OHC$. Thus, given an optimal $i$-vertex path contained in the optimal $(i+1)$-vertex paths, the probability that it is in the $OHC$ is $s = \frac{i+1}{{{i+1}\choose{2}}} = \frac{2}{i}$, and the probability that it is contained in an optimal $(i+1)$-vertex path outside the $OHC$ is $t = \frac{i-2}{i}$. Given an optimal $i$-vertex path containing $e$, this path may or may not lie within the $OHC$. The (comprehensive) probability that it is contained in the optimal $(i+1)$-vertex paths is $p = qt + rs = \frac{1}{i-1}\times \frac{i-2}{i} + \frac{2}{i-1}\times \frac{2}{i} = \frac{i+2}{i(i-1)} > q$. If $r\geq \frac{2}{i-1}$ is considered, $p \geq \frac{i+2}{i(i-1)}$ holds. 
	
	According to $p_i(e)$, the number of optimal $i$-vertex paths containing $e$ is computed as $N_i=(i-1){{i}\choose{2}}p_i(e)$. Each of them will be contained in the optimal $(i+1)$-vertex paths with the (comprehensive) probability $p > q$. In addition, there is one optimal $(i+1)$-vertex path where $A$ and $B$ are the endpoints. This optimal $(i+1)$-vertex path does not include $e = (A,B)$. Neither do the two optimal $i$-vertex paths embedded within it. As this optimal $(i+1)$-vertex path is neglected, the probability $p_{i+1}(e)$ is computed as formula (\ref{eq4}). $p_{i+1}(e) = p_i(e)$ indicates that $p_i(e)$ keeps the same value from $i$ to $i+1$. As $r \geq 2q$ is considered for $e\in OHC$, $p_{i+1}(e) \geq p_i(e)$ holds. It means that $p_i(e)$ increases or keeps the same value from $i$ to $i+1$. 
	
	\begin{equation}
		p_{i+1}(e) = \frac{N_ip}{{{i+1}\choose{2}} - 1} = p_i(e) 
		\label{eq4}
	\end{equation}
	
	If this optimal $(i+1)$-vertex path is considered, the probability $p_{i+1}(e) \geq \left(1 - \frac{2}{i(i+1)}\right)p_i(e)$ is computed. $p_{i+1}(e)$ is allowed to have a slight decrement from $i$ to $i+1$ if $p_i(e)$ becomes smaller. As $i$ is big enough, $\frac{2}{i(i+1)}$ 
	tends to zero. Meanwhile, $p_{i+1}(e) \geq p_i(e)$ holds, and $p_i(e)$ increases from $i$ to $i+1$. 
	
	Now, the change of $p_k(e)$ according to $k$ will be predicted for $e = (A,B)$ where $k\in [4, i]$ and $i\in [4,n-1]$. Given a $K_i$ containing $e$, the optimal $(i-1)$-vertex paths containing $e$ are included in the $OHC$ of the $K_i$ with the equal probability, and they are contained in the optimal $i$-vertex
	paths outside the $OHC$ in the $K_i$ with another equal probability. Thus, the formulae (\ref{eq3}) and (\ref{eq4}) work for $e$ from $i-1$ to $i$. It indicates that the probability $p_k(e)$ decreases or maintains the same value from $i$ to 4 where $k\in [4,i]$. Because $p_4(e)\geq \frac{1}{2}$ is proven based on Theorem \ref{th001}, it says that $p_k(e)\geq \frac{1}{2}$ holds according to $k$. It implies that $e$ is contained in more than half of the optimal $k$-vertex paths in a $K_k$ containing it on average. Because $p_{i+1} \geq p_i(e)$ as $k = i$, $p_{i+1}(e) \geq \frac{1}{2}$ also holds. Thus, the lower frequency bound of $e$ is $\frac{1}{2}{{i+1}\choose{2}}$ based on the ${{i+1}\choose{2}}$ $OP^{i+1}$s in the $K_{i+1}$. 
	
	As $i + 1 = n$, the probability $p_k(e) \geq \frac{1}{2}$ holds for $e\in OHC$ in $K_n$. Because $p_k(e)$ implies the average case according to the $K_k$s containing $e$, $e$ is contained in more than $\frac{1}{2}{{k}\choose{2}}$  optimal $k$-vertex paths in a $K_k$ containing $e$. 
	
	In the next, the relationship between $p_i(e)$ and $q_i(e)$ is given. Because $N_i=M_i + i-1$, the formula
	(\ref{eq5}) is derived for $q_i(e)\leq 1$. Thus, $p_{i+1}(e) = p_i(e) \geq q_{i+1}(e) = q_i(e)$ holds for $e\in OHC$ in the $K_{i+1}$. Moreover, as $i$ is big enough, $\frac{2}{i(i-1)}$ tends to zero, and $p_i(e) = q_i(e)$ appears. 
	
	\begin{equation}
		p_i(e) = \left[1 - \frac{2}{i(i-1)}\right]q_i(e) + \frac{2}{i(i-1)} \geq q_i(e)
		\label{eq5}
	\end{equation}
	
	For the ordinary edge $g = (A,C)$, it is not contained in the $OHC$ of $K_{i+1}$. Thus, it is not contained in the $OP^i$s and $OP^{i+1}$s in the $OHC$. The relationships between $p_{i+1}(g)$ and $p_i(g)$ is given as formula (\ref{eq6}). It indicates that the probability $p_i(g)$ decreases from $i$ to $i + 1$. Thus, the ordinary edges will reach the maximum probability $p_k(e)$ as $k$ is small. As $i$ is big enough, $p_{i+1}(e)\approx p_i(e)$ holds because $\frac{2}{i+1}$ tends to zero. It indicates that $p_i(e)$ decreases gradually according to $i$ for big $i$. \end{proof}	
	
	\begin{equation}
		p_{i+1}(g) = \frac{N_iq}{{{i+1}\choose{2}}} = \left[1 - \frac{2}{i+1}\right]p_i(g) < p_i(g)
		\label{eq6}
	\end{equation}

\begin{theorem}
	\label{th2}
	Given a $K_i$  containing ${{i}\choose{2}}$ $OP^i$s where $i\in[4,n]$, the frequency of an ordinary edge outside the $OHC$ is smaller than  $\frac{1}{2}{{i}\choose{2}}$ in
	the corresponding frequency $K_i$, and the expected frequency is smaller than $\frac{i+2}{2}$. 
\end{theorem}

\begin{proof}
	Given a vertex $v$ in the frequency $K_i$, the total frequency of the $i-1$ associated edges is $(i-1)^2$ based on the ${{i}\choose{2}}$ $OP^i$s. Thus, there are no four edges containing $v$ each of whose frequency is bigger than $\frac{1}{2}{{i}\choose{2}}$. Otherwise, the total frequency of the four edges is bigger than the total frequency $(i-1)^2$ of the $i-1$ edges containing $v$. 
	
	Provided that there are three edges $e_1$, $e_2$ and $e_3$ containing $v$ in the frequency $K_i$, each of their frequencies is bigger than $\frac{1}{2}{{i}\choose{2}}$. Moreover, $e_1$ and $e_2$ are the $OHC$ edges, and $e_3$ is one ordinary edge. Let $p_i(e_j) \geq \frac{1}{2}$ denote the probability that $e_j$ ($j$ = 1, 2, 3) is contained in the $OP^i$s in the $K_i$. When the $OP^{i-1}$s containing $v$ are used to compute the $OP^i$s by adding one more vertex, each of the $OP^{i-1}$s is contained in the $OP^i$s with the equal probability on average. Based on Theorem \ref{th002}, $e_1$ and $e_2$ will preserve the big probability $p_k(e_1) \geq \frac{1}{2}$ and $p_k(e_2) \geq \frac{1}{2}$ according to $k\in[4, i -1]$ based on the optimal $k$-vertex paths. Given a frequency $K_k$ containing $e_1$ and $e_2$, the frequency of $e_1$ and $e_2$ will be bigger than $\frac{1}{2}{{k}\choose{2}}$ on average. 
	
	Since $e_3$ is not the $OHC$ edge of the $K_i$, $p_i(e_3) = \left[1 - \frac{2}{i}\right]p_{i-1}(e_3)$ is computed based on formula (\ref{eq6}). It indicates that $p_i(e_3)$ increases from $i$ to $i-1$. Since $p_i(e_3) \geq \frac{1}{2}$, $p_{i-1}(e_3) \geq \frac{1}{2}$ holds. If $e_3$ becomes one $OHC$ edge in some $K_{i-1}$s, $p_{i-2}(e_3) = p_{i-1}(e_3) > p_i(e_3) \geq \frac{1}{2}$ can be derived based on the formula (\ref{eq4}) with respect to these frequency $K_{i-1}$s. If $e_3$ maintains the ordinary edge in the $K_{i-1}$s, $p_{i-2}(e_3) > p_{i-1}(e_3) \geq \frac{1}{2}$ holds based on formula (\ref{eq6}). For $e_3$, it may be one $OHC$ edge or an ordinary edge in each of the $K_k$s containing it as $k \in [4,i-1]$. Thus, $p_k(e_3)$ increases according to $k$ from $i$ to $4$ based on formulae (\ref{eq6}) and (\ref{eq4}). It indicates that $p_k(e_3) \geq \frac{1}{2}$ holds for $e_3$ according to $k$. Given a frequency $K_k$ containing $e_3$, $e_3$ has a frequency bigger than $\frac{1}{2}{{k}\choose{2}}$.  
	
	In the end, $e_1$, $e_2$ and $e_3$ are contained in one frequency $K_4$ where the probability $p_4(e_j) \geq \frac{1}{2}$ exists for each of them, respectively. However, the three adjacent edges must have the frequency 1, 3 and 5 in the frequency $K_4$, see the six frequency $K_4$s for a $K_4$ in Figure \ref{quadrilaterals}. There must be one edge having the frequency $1 < 3$. Thus, the three edges containing $v$ cannot have the probability $p_i(e) \geq \frac{1}{2}$ at the same time. Since $e_1$ and $e_2$ are the $OHC$ edges in the $K_i$, $p_i(e_1)\geq \frac{1}{2}$ and $p_i(e_2) \geq \frac{1}{2}$ hold based on Theorem \ref{th002}. Moreover, $e_3$ and the four adjacent $OHC$ edges of $K_i$ are contained in $4(i-3)$ frequency $K_4$s where $e_3$ has the frequency 1 with the probability $\frac{1}{2}$. It indicates $e_3$ will have the frequency 1 in many other frequency $K_4$s containing it. Thus, $p_i(e_3) < \frac{1}{2}$ exists. One sees $f(e)\geq \frac{1}{2}{{i}\choose{2}}$ and $p_i(e) \geq \frac{1}{2}$ is the sufficient and necessary conditions for an $OHC$ edge in a $K_i$ based on the frequency $K_i$.

	Because $f(e_1) \geq \frac{1}{2}{{i}\choose{2}}$ and $f(e_2) \geq \frac{1}{2}{{{i}\choose{2}}}$, the average frequency of the $i-3$ ordinary edges $e_3$ containing $v$ will be smaller than $i-3$ in the  frequency $K_i$. 
	As each of the two $OHC$ edges has the lowest frequency $f(e) = \frac{1}{2}{{i}\choose{2}}$,
	the average frequency of the $i-3$ ordinary edges is $\frac{1}{i-3}{{i-1}\choose{2}} \leq \frac{i+2}{2}$ if $i\geq 4$. If $e_3$ has each of the frequencies on the $i-3$ edges with the equal probability, the expected frequency of $e_3$ is smaller than $\frac{i+2}{2}$. 
	
	Thus, most ordinary edges will have a frequency smaller than $\frac{i+2}{2}$ in the frequency $K_i$. As $i$ becomes big, many ordinary edges have the frequency of 1 or zero in the frequency $K_i$. As they can not be used to build the $OP^i$s, these ordinary edges will be replaced by the $OHC$ edges and some other ordinary edges, see Theorem \ref{th01}. Thus, the frequency of an $OHC$ edge will increase and be bigger than $\frac{1}{2}{{i}\choose{2}}$ in the frequency $K_i$. 
\end{proof}
	Given one frequency $K_i$ ($i\in[4,n]$), the frequency of an $OHC$ edge is bigger than $\frac{1}{2}{{i}\choose{2}}$, whereas that of an ordinary edge is not. It indicates that $f(e) \geq \frac{1}{2}{{i}\choose{2}}$ is the sufficient and necessary conditions for $OHC$ edges in the $K_i$. Based on Theorem \ref{th002}, $p_k(e) \geq \frac{1}{2}$ holds for $e\in OHC$ of $K_i$ based on the frequency $K_k$s. It indicates that $e$ is contained in more than $\frac{1}{2}{{k}\choose{2}}$ $OP^k$s in a $K_k$ containing it. Thus, an $OHC$ edge of $K_i$ is also one $OHC$ edge of a $K_k$ containing it. 
	
	Moreover, the expected frequency for an $OHC$ edge in $K_i$ is much higher than that for an ordinary edge in the average case. 
\begin{theorem}
	\label{th22}
	Given a frequency $K_i$ $(i\in[4,n])$ containing ${{i}\choose{2}}$ $OP^i$s, the expected frequency for an $OHC$ edge is bigger than $\frac{i^2-4i+7}{2}$, whereas an ordinary edge has the expected frequency and maximum frequency smaller than 2 and $2(i-3)$, respectively. 
\end{theorem}

\begin{proof}	
	Given a frequency $K_i$, $e$ and $g$ denote an $OHC$ edge and ordinary edge, respectively. Based on Theorems \ref{th002} and \ref{th2}, $e$ has the frequency $f(e)= \frac{1}{2}{{i}\choose{2}}$ in the worst average case. On the other hand, $g$ has the frequency $f(g) = i-3$ in the best average case according to Theorem \ref{th01}. Because there are ${{i}\choose{2}}$ $OP^i$s in the frequency $K_i$, the probability that $e$ and $g$ is contained in an $OP^i$ is $p_i(e)\geq \frac{1}{2}$ and $p_i(g)\leq \frac{i-3}{{{i}\choose{2}}}$, respectively. 
	Thus, $\frac{p_i(e)}{p_i(g)} \geq \frac{i(i-1)}{4(i-3)} > \frac{i+2}{4}$ holds. It indicates that $f(e) > \frac{i+2}{4}f(g)$ holds. 
	
	In the $K_{i-1}$s containing  $g$, $g$ will be the $OHC$ edge. As each $OP^{i-1}$ in the $K_i$ is contained in the $OP^i$s with the equal probability, $\frac{p_{i-1}(e)}{p_{i-1}(g)} > \frac{i+2}{4}$ holds based on Theorem \ref{th002}. It says that the number of $OP^{i-1}$s containing $e$ is more than $\frac{i+2}{4}$ times of that of $OP^{i-1}$s containing $g$. As the $OP^k$s ($k\in[4,i-1]$) containing $e$ and $g$ are considered, $\frac{p_k(e)}{p_k(g)} > \frac{i+2}{4}$ also holds.  If $i=4$, $\frac{p_4(e)}{p_4(g)} > \frac{3}{2}$ holds. It is known that $p_4(e) > 2p_4(g)$ based on Theorem \ref{th01}. Thus, $\frac{p_{i}(e\in OP^i)}{p_{i}(g\in OP^i)} > \frac{i+2}{4}$ is relaxed for small $i$. 
	
	We assume that the ordinary edges $g$ are contained in the $OP^i$s with the equal probability as well as the $OHC$ edges $e$ do. For a vertex $v$, if one $g$ containing $v$ appears once in an $OP^i$, each of the two adjacent $OHC$ edges $e$  will appear at least $\frac{i+2}{4}$ times in the $OP^i$s. Based on Theorem \ref{th01}, the total frequency of the two $OHC$ edges $e$ is bigger than $\frac{i+2}{2}$ times of that of all the $i-3$ ordinary edges $g$ containing $v$. In this case, if an $OP^i$ contains one ordinary edge $g$, the $OP^i$ will contain more than $\frac{i+2}{2}$ $OHC$ edges $e$ in $OHC$. 
	Because each $OP^i$ contains $i-1$ edges, it will contain more than $i-3$ $OHC$ edges $e$ and less than two ordinary edges $g$. 
	
	It is known that the $OHC$ includes $i$ $OP^i$s each of which contains none of the ordinary edges $g$. Thus, an $OHC$ edge $e$ is contained in more than $\frac{i^2-4i+7}{2}$ $OP^i$s on average. 
	On the other hand, the expected frequency $f(g) = \frac{2(i-1)^2}{(i+3)(i-2)}$ for an ordinary edge $g$ will tend to 2 as $i$ is big enough. If the frequency of the other ordinary edges containing $v$ is zero, the maximum frequency of $g$ will tends to $2(i-3)$. 
	
\end{proof}

\section{The method to identify the ordinary edges}
\label{sec5}

The dynamic programming algorithm for $TSP$ can be improved based on Theorem \ref{th002}. 
\begin{theorem}
\label{th003}
	Based on the inclusions between the $OP^i$s in $K_n$, the $OHC$ can be found in $O(n^6 2^{\frac{n}{2}})$ time using dynamic programming. 
\end{theorem}
\begin{proof}
	Based on Theorem \ref{th002}, an $OHC$ edge $e$ will obtain the biggest frequency $f(e)$ at $i_0 = \frac{n}{2} + 2$ for even $n$ or $\frac{n+1}{2} + 1$ for odd $n$. After $i > i_0$, $f(e)$ will decrease according to $i$. Given an ordinary edge $g$, $f(g)$ will also obtain the biggest value of $f(g)$ at $i_0$. However, $f(g) < f(e)$ must hold at $i = i_0 + 1$. Otherwise, $f(g) \geq \frac{1}{2}{{n}\choose{2}}$ will appear at $i = n$ since we assume each $OP^i$ is contained in the $OP^{i+1}$s with the equal probability. Thus, it is necessary to compute the frequency of each edge at $i = i_0 + 1$, and compare them for finding the $OHC$. 
	
	Based on the law of large numbers in probability theory, one can choose a constant number $c\ll n$ of $K_{i_0+1}$s containing an edge to compute its frequency. As the dynamic programming algorithm is adopted, it requires $O(n^42^{\frac{n}{2}})$ time to compute the ${{i_0+1}\choose{2}}$ $OP^{i_0+1}$s in a $K_{i_0+1}$. The computation time to compute the frequencies of the ${{n}\choose{2}}$ edges in $K_n$ is $O(cn^62^{\frac{n}{2}})$. Since $c$ is a constant, it indicates that the computation time to resolve $TSP$ is $O(n^62^{\frac{n}{2}})$. 
\end{proof}
	
Although the dynamic programming algorithm for $TSP$ is improved, it is still not applicable as $n$ becomes big and large. In the following, we shall give an  applicable method to identify the ordinary edges. Based on Theorem \ref{th002}, the probability $p_i(e) \geq \frac{1}{2}$ and the average frequency $f(e) \geq \frac{1}{2}{{i}\choose{2}}$ hold for an edge $e\in OHC$ in $K_n$ based on the frequency $K_i$s where $i\in [4,n]$. It means that $e$ will have a frequency bigger than  $\frac{1}{2}{{i}\choose{2}}$ in a frequency $K_i$ containing it on average. Based on Theorem \ref{th2}, $e$ is also the $OHC$ edge of the $K_i$ containing it. Thus, given a $K_i$ containing $e$, $e$ is the $OHC$ edge of the $K_i$, $f(e)\geq \frac{1}{2}{{i}\choose{2}}$ and $p_i(e) \geq \frac{1}{2}$ hold based on the $OP^i$s in the $K_i$. 

For most ordinary edges $g$ in $K_n$, $p_i(g) < \frac{1}{2}$  and $f(g) < \frac{1}{2}{{i}\choose{2}}$ exist based on the frequency $K_i$s although $i$ is a small number. These ordinary edges can be identified based on the frequency $K_i$s in the polynomial computation time. For some special ordinary edges $g$ in $K_n$, $p_i(g) \geq \frac{1}{2}$ and $f(g) \geq \frac{1}{2}{{i}\choose{2}}$ also appear as $i$ is small, such as $i\ll n$. It is difficult to distinguish these ordinary edges from $OHC$ edges by edge frequency and probability based on the frequency $K_i$s for small number $i$. The changes of the average frequency and probability for the $OHC$ edges and ordinary edges are demonstrated in paper \cite{DBLP:journals/Wang26}. Here we give another method to identify the ordinary edges for saving computation time. 

An ordinary edge $g$ is contained in ${{n-2}\choose{i-2}}$ $K_i$s in $K_n$. These $K_i$s are classified into several kinds according to the number of $OHC$ edges in $K_n$ adjacent to $g$. (a) There is none of the $OHC$ edges of $K_n$ in the $K_i$. In the $K_i$, $g$ may be one $OHC$ edge of the $K_i$. (b) There is one $OHC$ edge of $K_n$ adjacent to $g$ in the $K_i$. $g$ may also be one $OHC$ edge in the $K_i$. (c) There are two $OHC$ edges in $K_n$ adjacent to $g$ in the $K_i$. The two $OHC$ edges of $K_n$ are adjacent to $g$ on one endpoint or on both endpoints. If the two $OHC$ edges are adjacent to $g$ on both endpoints, each $OHC$ edge is adjacent to $g$ on one endpoint, respectively. In  this case, $g$ may also be one $OHC$ edge of the $K_i$. Otherwise, $g$ will be an ordinary edge in the $K_i$. (d) There are three or four $OHC$ edges in $K_n$ adjacent to $g$ in the $K_i$. $g$ is an ordinary edge in the $K_i$.

Given a $K_i$ containing $g$ and two $OHC$ edges $e_1$ and $e_2$  of $K_n$ adjacent to $g$ on the same vertex, $e_1$ and $e_2$ will be two $OHC$ edges of the $K_i$ based on Theorems \ref{th002} and \ref{th2}. Since there is only one $OHC$ in the $K_i$ and there are only two $OHC$ edges associated with each vertex, $g$ must be one ordinary edge of the $K_i$. In this case, $g$ will have a frequency smaller than $\frac{1}{2}{{i}\choose{2}}$ based on the ${i}\choose{2}$ $OP^i$s in the $K_i$. 

In Figure \ref{OHC}, $g$ is adjacent to two pairs of the $OHC$ edges in $K_n$ where each pair of the $OHC$ edges is adjacent to $g$ on one endpoint, respectively. In $K_n$, $g$ and each pair of the $OHC$ edges is contained in ${{n-4}\choose{i-4}}$ $K_i$s. In each of the $K_i$s, $g$ is an ordinary edge, and it is contained in smaller than $\frac{1}{2}{{i}\choose{2}}$ $OP^i$s. Since $g$ contains two vertices, there are total $K = 2{{n-4}\choose{i-4}} - {{n-6}\choose{i-6}}$ frequency $K_i$s where $f(g) < \frac{1}{2}{{i}\choose{2}}$ and $p_i(g) < \frac{1}{2}$ exist. In each of the other frequency $K_i$s, $g$ may be one $OHC$ edge or an ordinary edge. In the extreme case, $g$ is an $OHC$ edge in each of the ${{n-2}\choose{i-2}} - K$ $K_i$s. In this case, the smallest percentage of the frequency $K_i$s where $f(g) < \frac{1}{2}{{i}\choose{2}}$ and $p_i(e) < \frac{1}{2}$ is computed as $r_i \in [0,1] $ in formula(\ref{eq7}). One sees that $r_i$ is a monotone function from $i = 6$ to $n$. Thus, the percentage of frequency $K_i$s where $g$ has a probability $p_i(e) < \frac{1}{2}$ monotone increases according to $i$. However, $p_i(e) \geq \frac{1}{2}$ for an $OHC$ edge $e$ increases or keeps the nearly equal value according to $i\in[4,n]$. It indicates that the percentage of frequency $K_i$s where $p_i(e) < \frac{1}{2}$ monotone decreases according to $i\in [6,n]$. Thus, the ordinary edges can be identified according to the $r_i$ computed in formula (\ref{eq7}). If $i = 4$, such an ordinary edge $g$ will be contained in more than $2(n-4)$ frequency $K_4$s where $f(g) = 1$ based on Theorem \ref{th001}. 

\begin{equation}
	r_i = \frac{K}{{{n-2}\choose{i-2}}} = \frac{2(i-2)(i-3)}{(n-2)(n-3)} - \frac{(i-2)(i-3)(i-4)(i-5)}{(n-2)(n-3)(n-4)(n-5)}
	\label{eq7}
\end{equation}

In Figure \ref{OHC}, if $v_2 = v_i$ or $v_k = v_n$, there are $n$ such ordinary edges $g = (v_1, v_j )$ in $K_n$. For such an ordinary edge $g$, $K = 2{{n-4}\choose{i-4}} - {{n-5}\choose{i-5}}$, and the smallest probability $r_i$ for $g$ is computed as formula (\ref{eq8}) which also monotone increases according to $i\in [5, n]$. If $i = 4$, such an ordinary edge $g$ is contained in more than $2(n-4)$ frequency $K_4$s where $f(g) = 1$ based on Theorem \ref{th001}.

\begin{equation}
	r_i = \frac{2(i-2)(i-3)}{(n-2)(n-3)} - \frac{(i-2)(i-3)(i-4)}{(n-2)(n-3)(n-4)}
	\label{eq8}
\end{equation}

Given an edge $g$ in $K_n$, there are several steps to compute the $r_i$ for identifying the ordinary edges. Step1: Choose a set of $M$ $K_i$s containing $g$. Step 2: Convert the $M$ $K_i$s into the frequency $K_i$s. Step 3: Compute the number of frequency $K_i$s where $f(g) < \frac{1}{2}{{i}\choose{2}}$ with respect to the $M$ frequency $K_i$s, and denote the number as  $M_i$.  Step 4: Compute the evaluation index $r_i$, i.e.,   $r_i = \frac{M_i}{M}$. Step 5: Let $i:= i + 1$, compute the $r_{i+1}$ according to the procedure in Step 1 $\sim$ 4. Step 6: If $r_i < r_{i+1}$, $g$ is one ordinary edge; Otherwise, it is not. 

If all $K_i$s and $K_{i+1}$s in $K_n$ are used to compute the $r_i$ and $r_{i+1}$ to find the ordinary edges, it will consume $O(n^i)$ computation time if $i$ is fixed. If $i$ is relatively big, it becomes impractical in real-world applications, especially for $TSP$ of big and large scale. In fact, most ordinary edges can be found based on Theorem \ref{th001}. If the average frequency of each edge is computed with the frequency $K_4$s, the edges with an average frequency smaller than 3 are ordinary edges. One can use the frequency $K_i$s for $i >4$ to find more ordinary edges with an average frequency smaller than $\frac{1}{2}{{i}\choose{2}}$ based on Theorem \ref{th002}. After many ordinary edges are found based on Theorems \ref{th001} and \ref{th002}, the residual ordinary edges will be identified using the index $r_i$. Note the graph containing the residual edges as $G=(V,E_r)$ where $E_r$ includes the $OHC$ edges and ordinary edges with an average frequency bigger than $\frac{1}{2}{{i}\choose{2}}$. The residual graph $G$ can be computed in $O(n^2log_2n)$ time, see \cite{DBLP:journals/Wang16}.

In real-life applications, only comparing two percentages $r_i$ and $r_{i+1}$ for an edge may not work well to find some ordinary edges $g$. The first reason comes from the random $K_i$s chosen for $g$. In some experiments, the values of $r_i$ and $r_{i+1}$ may deviate from the expected values although this occurrence  probability is very small. If this case happens, $r_i \geq r_{i+1}$ may appear in the experiments. Thus, it is better to do several experiments for evaluation.  Secondly, the number of adjacent $OHC$ edges is much smaller than that of the adjacent ordinary edges for $g$. If $i$ is small, there will be no $OHC$ edges in most of the selected $K_i$s. In this case, $g$ will maintain the big $r_i$ and it does not decrease from $i$ to $i+1$. If let $r_i < r_{i+1}$, we have to use more $K_i$s containing one or two pairs of the adjacent $OHC$ edges for $g$. The $K_i$s constructed by the edges in the residual graph $G = (V, E_r)$ just meet the requirements. In the residual graph $G$, each vertex is adjacent to two $OHC$ edges and a small number of ordinary edges. Even though the edges are selected at random to build the $K_i$s for $g$, one or two pairs of the adjacent $OHC$ edges have more chance to be selected for constructing the $K_i$s for $g$. Thus, $r_i$ will become bigger based on these frequency $K_i$s and $K_{i+1}$s. In addition, the bigger the $i$ is, the better the index $r_i$ works well to identify the ordinary edges $g$. 

\section{Conclusion}
\label{sec6}
The frequency $K_i$s ($i\in [4,n]$) are studied for characterizing the structures of the $OHC$ edges for symmetric $TSP$. $OHC$ edges illustrate the much higher frequency and probability than that of ordinary edges computed with the frequency $K_i$s. Firstly, an $OHC$ edge in $K_i$ has the frequency bigger than $\frac{1}{2}{{i}\choose{2}}$ in the  frequency $K_i$, whereas an ordinary edge has the frequency smaller than $\frac{1}{2}{{i}\choose{2}}$. On average, the expected frequency of an $OHC$ edge is bigger than $\frac{i^2-4i+7}{2}$ whereas an ordinary edge has the expected frequency smaller than 2 in the  frequency $K_i$. Secondly, as the frequency of each edge in $K_n$ is computed with the frequency $K_i$s, an $OHC$ edge in $K_n$ has the average frequency bigger than $\frac{1}{2}{{i}\choose{2}}$. It indicates that an $OHC$ edge in $K_n$ is the $OHC$ edge of a $K_i$ containing it on average. 

Based on the frequency $K_i$s, the probability that an $OHC$ edge has the frequency bigger than $\frac{1}{2}{{i}\choose{2}}$  increases according to $i\in[4,n]$. On the other hand, the probability that an ordinary edge has the frequency smaller than $\frac{1}{2}{{i}\choose{2}}$  increases according to $i\in[4,n]$. Based on the findings, a method to identify the ordinary edges is provided. This method is also discussed with respect to real-world applications. In the future, we shall design the algorithm for resolving $TSP$ based on the method .  

\section{Acknowledgements}
\label{sec7}
The author wants to thank the Fundamental Research funds for the Central Universities (No.2024JC006). The work benefits from the State Key Laboratory of Alterate Electrical Power System with Renewable Energy Sources (NCEPU).

\section{Declaration of competing interests}
\label{sec8}
The author declared that they had no affiliations with or involvement in any organization or entity with any financial interest in the subject matter or materials discussed in the manuscript. 
\section{Declaration of generative AI in scientific writing}
\label{sec9}
The author declared that there were no use of generative AI and AI-assisted technologies in scientific writing in the manuscript. 
\appendix
\section{Six frequency $K_4$s for a $K_4$}
\label{app1}

The six frequency $K_4$s for a $K_4$ are computed as follows.
Given a $K_4$ on four vertices $A$, $B$, $C$ and $D$, the layout of the four vertices is illustrated as that in Fig. \ref{quadrilaterals} (a). There are six edges $(A,B)$, $(A,C)$, $(A,D)$, $(B,C)$, $(B,D)$, and $(C,D)$ in the $K_4 = ABCD$. The distances of the six edges are denoted as $d(A,B)$, $d(A,C)$, $d(A,D)$, $d(B,C)$, $d(B,D)$ and $d(C,D)$, respectively. The six $OP^4$s in $K_4$ are up to the ordering of the three distance sums $S_1 = d(A,B)+d(C,D)$,  $S_2 = d(A,D)+d(B,C)$ and $S_3 = d(A,C)+d(B,D)$. For example, if $S_1 < S_2 < S_3$ is considered,
the six $OP^4$s are computed as $(A,D,C,B)$, $(A,B,D,C)$, $(A,B,C,D)$, $(B,A,D,C)$, $(B,A,C,D)$ and $(C,B,A,D)$. We enumerate the number of $OP^4$s containing an edge, and this number represents the frequency of the edge. For example, $(A,B)$ is contained in five of the $OP^4$s so it has the frequency of 5. As the frequency of each edge is computed with the six $OP^4$s, the frequency $K_4$ is illustrated in Figure \ref{quadrilaterals} (a). The inequality related to the three distance sums is given below the frequency $K_4$. 
Since the three distance sums $S_1$, $S_2$ and $S_3$ have the other five orderings, one can derive the corresponding set of   $OP^4$s and frequency $K_4$s. The other five frequency $K_4$s, and each ordering of the distance sums
are illustrated in Figure \ref{quadrilaterals} (b) $\thicksim$ (f), respectively.

\begin{figure}
	\centering
	\includegraphics[width=4.0in,bb=0 0 600 350]{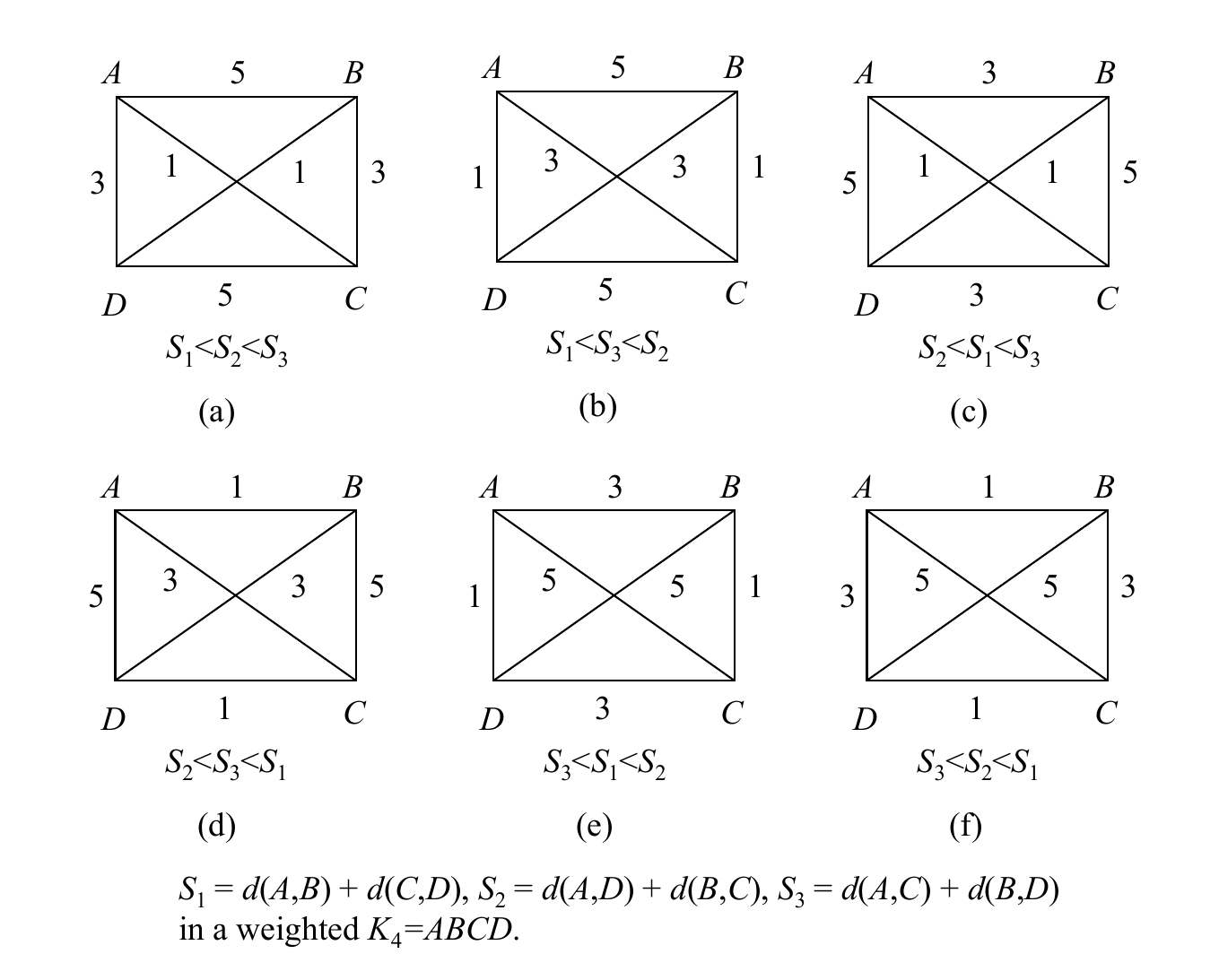}
	\caption{The six frequency $K_4$s for a weighted $K_4=ABCD$.}
	\label{quadrilaterals}
\end{figure}

In each frequency $K_4$, the frequency
of an edge is 1, 3 or 5. It implies that the edge is contained in 1, 3 or 5 $OP^4$s. Moreover, the frequency of $OHC$ edges in $K_4$ is either 3 or 5. For example, the $OHC$ edge $(A,B)$ has the frequency of 3 or 5 in (a), (b), (c), and (e). This means that the minimum frequency of $OHC$ edges in $K_4$ is 3. However, the expected frequency for an $OHC$ edge will be bigger than 3. Given an $OHC$ edge, such as $(A,B)$, it has the frequency of 3 in (c) and (e), and the frequency of 5 in (a) and (b). Thus, the expected frequency for $(A,B)$ is 4 with respect to the four frequency $K_4$s. Since there are six $OP^4$s in $ABCD$, the probability that $(A,B)$ is contained in an $OP^4$ in $ABCD$ is $\frac{2}{3}$. In addition, the two vertex-disjoint edges, such as $(A,B)$ vs $(C,D)$, $(A,C)$ vs $(B,D)$ and $(A,D)$ vs $(B,C)$, have the same frequency in a  frequency $K_4$. On the other hand, the three edges containing a vertex have three  different frequencies 1, 3 and 5, respectively.

\end{document}